\documentclass[11pt]{amsart}
\usepackage[latin1]{inputenc}
\usepackage{amsfonts}
\usepackage[toc,page,title,titletoc,header]{appendix}
\usepackage{graphicx,psfrag,epsfig,multirow,caption,subcaption}
\usepackage{amssymb,amsmath,amscd,amsthm,amssymb,verbatim,setspace}
\numberwithin{equation}{section}
\usepackage{mathrsfs,wrapfig}
\usepackage{indentfirst}
\usepackage{extarrows}
\usepackage{float}

\usepackage{xcolor}

\usepackage[colorlinks=true]{hyperref}

\newtheorem{theorem}{Theorem}[section]
\newtheorem{definition}[theorem]{Definition}
\newtheorem{conjecture}[theorem]{Conjecture}
\newtheorem{proposition}[theorem]{Proposition}
\newtheorem{lemma}[theorem]{Lemma}
\newtheorem{remark}[theorem]{Remark}
\newtheorem{corollary}[theorem]{Corollary}

\newtheorem*{remark*}{Remark}

\numberwithin{equation}{section}

\newcommand{\PiZ}{\Pi_{=0}}
\newcommand{\PiNZ}{\Pi_{\ne0}}

\newcommand{\Err}{\mathcal E}
\newcommand{\Rem}{\mathcal Q}

\newcommand{\dv}{\text{div}}

\newcommand{\mfk}{\mathfrak}

\newcommand{\eps}{\varepsilon}

\newcommand{\rmM}{\mathrm{M}}

\newcommand{\rmW}{\mathrm{W}}
\newcommand{\rmT}{\mathrm{T}}

\newcommand{\dt}{\delta}

\newcommand{\bn}{\mathbf{n}}
\newcommand{\br}{\mathbf{r}}

\newcommand{\bM}{\mathbf{M}}
\newcommand{\bS}{\mathbb{S}}

\newcommand{\bx}{\mathbf{x}}
\newcommand{\by}{\mathbf{y}}
\newcommand{\RR}{\mathbb{R}}
\newcommand{\mbfA}{\mathbf{A}}

\newcommand{\cA}{\mathcal{A}}
\newcommand{\cP}{\mathcal{P}}
\newcommand{\cM}{\mathcal{M}}
\newcommand{\cH}{\mathcal{H}}
\newcommand{\cE}{\mathcal{E}}
\newcommand{\cB}{\mathcal{B}}
\newcommand{\cS}{\mathcal{S}}
\newcommand{\cC}{\mathcal{C}}
\newcommand{\cL}{\mathcal{L}}
\newcommand{\cK}{\mathcal{K}}

\newcommand{\cQ}{\mathcal{Q}}

\newcommand{\cG}{\mathcal{G}}
\newcommand{\cU}{\mathcal{U}}
\newcommand{\cV}{\mathcal{V}}

\newcommand{\R}{\mathbb{R}}

\newcommand{\A}{\mathbb{A}}
\newcommand{\pr}{\partial}

\DeclareMathOperator{\loc}{loc}

\title[MCF with cylindrical singularities]
{Mean curvature flows with cylindrical singularities II: stability and genericity}
\author{Ao Sun}
\address{Lehigh University, Department of Mathematics, Chandler-Ullmann Hall, Bethlehem, PA 18015}
\email{aos223@lehigh.edu}
\author{Jinxin Xue}
\address{New Cornerstone Science Laboratory, Department of Mathematics, Rm A115, Tsinghua University, Haidian District, Beijing, 100084}
\email{jxue@tsinghua.edu.cn}
\date{\today}

\begin{document}
	\maketitle
	\begin{abstract}
		In \cite{SunXue25_GenericI}, we introduced the notion of nondegenerate cylindrical singularity using the ideas from dynamical systems. In this paper, we further study the properties of these singularities, showing that they are stable under small perturbations. Furthermore, we show that we can perturb some mean curvature flow with degenerate singularities to produce nondegenerate singularities.
	\end{abstract}
	
	\section{Introduction}
	Mean curvature flow (MCF) is defined as a family of hypersurfaces $\{\bM_t\}_{t\in I}$ evolving in $\mathbb{R}^{n+1}$ according to the equation $\pr_t x=\vec{H}(x)$, where $\vec H$ is the mean curvature vector. It is a fundamental geometric flow that has attracted considerable attention in diverse fields, including geometry, partial differential equations, and applied mathematics.
	
	The cylindrical singularities characterize the change of geometry and topology of mean curvature flows. On the other hand, it can have complicated shapes. In the famous example of ``marriage ring'' -- a rotationally symmetric mean curvature flow that is a thin torus, the singular set is a circle, a continuous curve. In \cite{SunXue25_GenericI}, we introduced the notion of {\bf nondegenerate cylindrical singularity}, and they are isolated in a spacetime backward parabolic neighborhood\footnote{Later in joint work with Zhihan Wang \cite{SunWangXue1_Passing}, we proved that the nondegenerate singularities are isolated in the forward parabolic neighborhood as well. We also derived many other properties of nondegenerate singularities, and a relatively complete picture of the mean curvature flow passing through them.}. Hence, the singularities of the marriage ring are not nondegenerate.
	
	Let us first recall the definition of nondegeneracy of a cylindrical singularity. Let $\cC_{n,k}:=S^{n-k}(\sqrt{2(n-k)})\times\R^k$ be a generalized cylinder for $k=1,2,\cdots,n-1$, and we use $\theta=(\theta_1,\cdots,\theta_{n-k+1})$ to denote the coordinates on the sphere factor and $y=(y_1,\cdots,y_k)$ to denote the coordinates on the $\R^k$ factor. We use $\varrho:=\sqrt{2(n-k)}$ to denote the sphere factor radius. On $\cC_{n,k}$, the natural functional space is the Gaussian weighted space, namely for $\Omega\subset\cC_{n,k}$ and $f:\Omega\to\R$, we define
	\[
	\|f\|_{L^2(\Omega)}^2=\int_{\Omega} |f(x)|^2 e^{-\frac{|x|^2}{4}}d\cH^n(x),
	\]
	and for any $\ell>0$,
	\[
	\|f\|_{H^\ell(\Omega)}^2=\int_{\Omega} \sum_{j=0}^\ell|\nabla^j f(x)|^2 e^{-\frac{|x|^2}{4}}d\cH^n(x).
	\]
	
	Suppose $\{\bM_{t}\}_{t\in I}$ is a mean curvature flow of embedded hypersurfaces in $\R^{n+1}$ and $M(\tau):=e^{\tau/2}\bM_{-e^{-\tau}}$ is the corresponding rescaled mean curvature flow detecting the singularity at $(0,0)$. We say $(0,0)$ is a cylindrical singularity if $M(\tau)\to \cC_{n,k}$ smoothly as $\tau\to\infty$ in any compact subset of $\R^{n+1}$. In \cite{SunXue25_GenericI}, based on the work of Colding-Minicozzi \cite{CM1}, we proved the following asymptotic theorem.
	\begin{theorem}[Theorems 1.2 and 1.3 in \cite{SunXue25_GenericI}]
		\label{ThmNF-sqrtt}
		Suppose that the rescaled mean curvature flow $M_\tau=e^{\tau/2}\bM_{-e^{-\tau}}$ converges to $\cC_{n,k}$ smoothly locally as $\tau\to\infty$. Then one of the following alternatives holds.
		
		\begin{enumerate}
			\item There is a nonempty subset $\mathcal I\subset\{1,\ldots,k\}$ such that, after a fixed rotation in the $\mathbb R^k$-factor, for every $K>0$ and $\vartheta\in(0,1)$,
			\[
			\left\|
			u(\cdot,\tau)
			-
			\frac{\varrho}{4\tau}
			\sum_{i\in\mathcal I}(y_i^2-2)
			\right\|_{H^1(\cC_{n,k}\cap B_{K\sqrt{\tau}})}
			\leq
			C\tau^{-1-\vartheta}
			\]
			and
			\[
			\left\|
			u(\cdot,\tau)
			-
			\varrho
			\left(
			\sqrt{
				1+\frac{\sum_{i\in\mathcal I}(y_i^2-2)}{2\tau}
			}
			-1
			\right)
			\right\|_{C^1(\cC_{n,k}\cap B_{K\sqrt{\tau}})}
			\leq
			C\tau^{-\vartheta}
			\]
			for all sufficiently large $\tau$.
			
			\item There are constants $K>0$ and $c>0$ such that
			\[
			\|u(\cdot,\tau)\|_{H^1(\cC_{n,k}\cap B_{K\sqrt{\tau}})}
			+
			\|u(\cdot,\tau)\|_{C^1(\cC_{n,k}\cap B_{K\sqrt{\tau}})}
			\leq
			Ce^{-c\tau}
			\]
			for all sufficiently large $\tau$.
		\end{enumerate}
	\end{theorem}
	
	\begin{definition}
		\label{DefDeg}
		A cylindrical singularity as in Theorem \ref{ThmNF-sqrtt} is
		
		\begin{itemize}
			\item \emph{nondegenerate} if $\mathcal I=\{1,\ldots,k\};$
			
			\item \emph{partially nondegenerate} if $\emptyset\neq\mathcal I\subsetneq\{1,\ldots,k\};$
			
			\item \emph{degenerate} if $\mathcal I=\emptyset.$
		\end{itemize}
	\end{definition}
	From a dynamical point of view, nondegenerate cylindrical singularities have
	the full-rank neutral asymptotic behavior.  The linearized equation for a graph
	over $\cC_{n,k}$ is
	\[
	\partial_\tau u=L_{\cC_{n,k}}u
	:=\Delta_{\cC_{n,k}}u-\frac12\langle x,\nabla u\rangle+u.
	\]
	The positive $L$-modes are generated by the constant function and by
	$\theta_\alpha,y_j$.  Convergence to the cylinder forces these modes to be
	removed by the choice of spacetime center and scale.  After the rotational
	zero modes have also been fixed, the first genuine asymptotic modes are the
	quadratic Hermite functions.  Nondegeneracy means that their coefficient
	matrix has full rank; after a fixed axial rotation its leading term is
	$\sum_{i=1}^k(y_i^2-2)$.
	
	The main results in this paper extend the genericity of the nondegenerate cylindrical singularities from the dynamical level to the flow level. First of all, we prove the stability of nondegenerate singularities.
	
	\begin{theorem}[Stability of nondegenerate singularities]\label{ThmStability}
		Assume $\{\bM_{t}\}_{t\in[-1,0)}$ is a mean curvature flow of closed embedded hypersurfaces in $\R^{n+1}$, $(0,0)$ is a nondegenerate cylindrical singularity modeled by $\cC_{n,k}$. Then there is an $\eps_0>0$ such that for all $u:\ \bM_{-1}\to\R$ with $\|u\|_{C^{2}}\leq \eps_0$, the perturbed mean curvature flow with initial condition $\widetilde\bM_{-1}=\mathrm{Graph}\{ x+ u(x)\mathbf n(x)\ |\ x\in \bM_{-1}\} $ admits a nondegenerate singularity in a spacetime neighborhood of $(0,0)$ modeled by $\cC_{n,k}$. 
	\end{theorem}
	
	We also have the following more quantative version of stability. 
	
	\begin{theorem}\label{ThmQuantativeStability}
		Let $M_{\tau}$, $\tau\in[0,\infty)$, be an RMCF converging to
		$\cC_{n,k}$ in $C^\infty_{\mathrm{loc}}$ and corresponding to a
		nondegenerate singularity of the MCF. There exist $\rho,\eps_1>0$ and
		$T>0$ such that, for every $\tau_0>T$, the following holds. Let
		$v_{\tau_0}$ be supported in $B_{\rho\sqrt{\tau_0}}$ and satisfy
		\[
		\|v_{\tau_0}\|_{H^1}
		\leq\eps_1\|\chi_{\rho\sqrt{\tau_0}}u_{\tau_0}\|_{H^1},
		\qquad
		\|v_{\tau_0}\|_{C^2(B_{\rho\sqrt{\tau_0}})}\leq\eps_1.
		\]
		Then, after a possible translation and dilation, the perturbed flow
		converges to $\cC_{n,k}$ and corresponds to a nondegenerate singularity
		of the MCF.
	\end{theorem}
	
	\begin{theorem}[Existence of model solutions]\label{ThmModel}
		Given $1\leq k< n$, there exists an $O(n-k+1)\times O(k)$-invariant mean curvature flow with a nondegenerate cylindrical singularity at the origin modeled by $\cC_{n,k}$.  
	\end{theorem}
	
	The last three theorems will be proved in Section \ref{S_stability}. Previously, Gang-Sigal \cite{GangSigal09_Neck} and Gang-Knopf \cite{GangKnopf15_UniversalityMCF} discussed the formation of neck singularities for mean curvature flows that are close to a neck profile. In particular, they proved that small perturbations will preserve these neckpinch singularities for those flows that are close to the cylinder with a neckpinch. 
	
	The stability of nondegenerate cylindrical singularities is a subtle matter. In fact, the generalized cylinders, $\cC_{n,k}$, are stable at the linear level, but not at the flow level. In \cite{CM1}, Colding-Minicozzi introduced a quantity called \emph{entropy} for hypersurfaces in $\R^{n+1}$, defined by
	\begin{equation*}
		\lambda(\Sigma)=\sup_{x_0\in \R^{n+1}, \tau_0\in(0,\infty)} \frac{1}{(4\pi \tau_0)^{n/2}}\int_{\Sigma}e^{-\frac{|x-x_0|^2}{4\tau_0}}d\cH^n(x).
	\end{equation*}
	This quantity is a Lyapunov functional if we view the mean curvature flow or the rescaled mean curvature flow as a dynamical system. Moreover, Colding-Minicozzi defined a linear stability called \emph{$L$-stability}, and proved that if a shrinker is entropy stable, then it must be $L$-stable. Finally, they prove that only the sphere and generalized cylinders are $L$-stable shrinkers. This is the first major achievement in the study of generic mean curvature flows.
	
	On the other hand, a degenerate cylindrical singularity may not be stable. For instance, the mean curvature flow starting from the ``peanut'' surface constructed in \cite{AAG} has a degenerate cylindrical singularity that can be perturbed into either a spherical or a cylindrical singularity depending on the chosen perturbations. The perturbation can produce a singularity whose singularity model with strictly lower entropy. Therefore, even if $\cC_{n,k}$ are stable as the singularity models, the singularities modeled by them may not be stable.
	
	Our second main result is the following perturbation theorem. Recall that a singularity $(p,T)$ is type-I if there exists a backward parabolic neighborhood such that inside this neighborhood, there exists $C>0$ so that $\bM_t$ satisfies $|A|^2\leq C(T-t)^{-1}$, where $A$ is the second fundamental form of the flow.
	
	\begin{theorem}[Denseness of nondegenerate singularities]\label{ThmGenericity}
		Assume $\{\bM_{t}\}_{t\in[-1,0)}$ is a mean curvature flow of closed embedded hypersurfaces in $\R^{n+1}$, $(0,0)$ is a degenerate cylindrical singularity modeled by $\cC_{n,k}$, and is a type-I singularity. 
		
		Then for any $\eps>0$ and $\delta>0$, there exists $f:\bM_{-1}\to\R$ with $\|f\|_{C^2}\leq \eps$, such that the perturbed mean curvature flow $\{\widetilde{\bM}_{t}\}$ starting from $\widetilde{\bM}_{-1}:=\{x+f(x)\bn:x\in\bM_{-1}\}$ has a singularity in a $\delta$-neighbourhood of $(0,0)$, that is modeled by $\cC_{n,k}$, and is nondegenerate.
	\end{theorem}

	The type-I assumption is a technical requirement in our proof. We expect that it is not necessary, but we do not have a proof here.
	
	On the other hand, the assumption is valid in particular for the example of the marriage ring, and more generally, ``thin tube". Recall that the marriage ring is a rotationally symmetric mean curvature flow in $\R^{n+1}$ whose singular set is given by a rotationally symmetric $S^1$, and each singularity is modeled by $\cC_{n,1}$. Conjecturally, for any closed $C^1$-curve $\gamma$, there exists a mean curvature flow $\{\bM_t\}_{t\in[0,t_0)}$ such that $\bM_{0}$ is the boundary of a tubular neighbourhood of $\gamma$, and $M_t$ collapses to a singular set $\gamma$. The authors heard the following conjecture from Colding-Minicozzi.
	\begin{conjecture}[Colding-Minicozzi]
		Any smooth curve $\gamma\subset\R^3$ can be the singular set of a $($mean convex$)$ mean curvature flow. 
	\end{conjecture}
	
	We refer to such a mean curvature flow as a \emph{thin tube}. White \cite[Section 5]{Wh2} conjectured that mean curvature flows in $\R^3$ can only have the singular set consisting of isolated singularities and curves.
	
	The singular set must have certain regularity. In  \cite{CM3}, Colding-Minicozzi proved that the singular set of cylindrical singularities in $\R^3$ is contained in a finite union of $C^1$-curves and countably many isolated points. In \cite{SunWangXue1_Passing}, joint with Zhihan Wang, we proved that the singular set of cylindrical singularities in $\R^3$ is contained in a $C^{2,\alpha}$-curve and countably many isolated points. Both \cite{CM3, SunWangXue1_Passing} have analogous results for higher-dimensional cases.
	
	We show that when the singular set is a smooth closed curve, even a small initial perturbation can lead to an isolated first-time singularity. This result holds true for various scenarios, including the example of a marriage ring. Heuristically, we can create an isolated singularity by slightly squeezing the marriage ring.
	
	\begin{theorem}\label{ThmMarriageRing}
		Suppose $\{\bM_t\}_{t\in[-1,0)}$ be a smooth embedded mean curvature flow in $\R^{n+1}$ and the singular set at time $0$ is a $C^{2,\alpha}$-closed curve $\gamma$, and the singularities are modeled by $\cC_{n,1}$. Then for any $p\in \gamma$ and any $\eps_0>0$ there exists a function $u_0\in C^{2,\alpha}(\bM_{-1})$ with $\|u_0\|_{C^{2,\alpha}}<\epsilon_0$ such that the perturbed MCF $\{\widetilde{\bM}_t\}$ starting from $\widetilde{\bM}_{-1}:=\{x+u_0(x)\bn(x):x\in\bM_{-1}\}$ has a single first-time singularity in a small neighborhood of $p$.
	\end{theorem}
	
	Finally, we discuss a possibly existing conjectural example called \emph{firecracker}. Imagine we can perturb the initial data of the mean curvature flow to produce a nondegenerate singularity from an existing degenerate singularity. If we see another degenerate singularity, we repeat this process. This process may never terminate. Motivated by this issue, we introduce the following notion:
	\begin{definition}\label{def:firecracker}
		A mean curvature flow $\bM$ is said to have a \emph{$($firecracker$)$} singular set, if the first-time singular set of $\bM$ contains a degenerate singularity (referred to as a ``firecracker'') that is the limit of a sequence of nondegenerate spacetime singularities.
	\end{definition}
	
	At this moment, there is no example of such a singular set. On the other hand, recently Simon \cite{Si} constructed examples of minimal hypersurfaces with any given prescribed compact singular set. It is conceivable that a parabolic version of Simon's construction would yield a firecracker singular set for mean curvature flow. On the other hand, the firecracker singular set seems to be a highly unstable object. Therefore, we make the following conjecture.
	
	\begin{conjecture}
		The firecracker singular set exists in mean curvature flow, but does not exist generically.
	\end{conjecture}
	
	It would also be interesting to compare the firecracker with another conjectural example in \cite{ChoiHaslhoferHershkovits21_NoteSelfLimitFlow}, where Choi-Haslhofer-Hershkovits conjectured that there can not be a sequence of spherical singularities converging to a cylindrical singularity.

	\subsection{Generic level set flows}
	Let us discuss another application of the genericity of nondegenerate singularities. Level set flow is a weak formulation of mean curvature flow, where the evolution of a hypersurface $\bM_t$ is represented by the level sets of a function $f(x,t)$ satisfying the equation
	\[
	\pr_t f=|\nabla f|\dv\left(\frac{\nabla f}{|\nabla f|}\right).
	\]
	The level set flow PDE was first proposed by Osher-Sethian \cite{OS} in the context of numerical analysis. Later, Chen-Giga-Goto \cite{CGG} and Evans-Spruck \cite{ES} proved the existence of a viscosity solution to the level set flow PDE. If $f(x,t)$ is a solution to the level set flow PDE and $|\nabla f(x,t)|\neq 0$, then $\bM_t:=\{x:f(x,t)=0\}$ is a smooth MCF when $\{\bM_t\}$ is a family of codimension $1$ hypersurfaces. Consequently, one can define a weak solution to MCF as the level sets of a level set flow function. When $\bM_0:=\{x:f(x,0)=0\}$ is mean convex, the MCF evolves monotonically. This allows one to introduce an \emph{arrival time} function $g(x):=f(x,t)+t$, where $g$ satisfies the equation
	$
	-1=|\nabla g|\dv\left(\frac{\nabla g}{|\nabla g|}\right).
	$
	
	The regularity of the LSF is a natural question in the context of PDE. Huisken \cite{Hu2} proved that the arrival time function of a convex hypersurface has at least $C^2$ regularity, and later Sesum \cite{Se} constructed examples with at most $C^3$ regularity. In our previous work \cite{SX3}, we proved that for a generic convex hypersurface, the arrival time function does not have $C^3$ regularity.
	
	When the initial hypersurface is mean convex but not convex, the regularity becomes more subtle due to the possible cylindrical singularities. Colding-Minicozzi \cite{CM4, CM5} studied the relationship between the structure of the singular set of mean curvature flow and the regularity of the arrival time function. Specifically, they showed that the regularity of the arrival time function must be at least $C^{1,1}$, and the regularity is $C^2$ if and only if the singular time is unique and the singular set is a $C^{2,\alpha}$-submanifold. In this paper, we prove a generic regularity result for the arrival time function.
	\begin{theorem}\label{thm:main thm LSF}
		Let $\cV$ be the space of mean-convex closed surfaces in $\R^{3}$ with the $C^r$ topology, $r\geq2$.  Assume that the sets $\cU_1$ and $\cU_2$ defined below are open and that $\cU_1\cup\cU_2$ is {\color{red}dense} in $\cV$.  Then $\cU=\cU_1\cup\cU_2$ has the following properties:
		\begin{itemize}
			\item for every initial surface in $\cU_1$, the arrival time is $C^2$ but not $C^3$, and the flow has a single isolated spherical singularity;
			\item for every initial surface in $\cU_2$, the arrival time is $C^{1,1}$ but not $C^2$, and the flow has an isolated nondegenerate cylindrical singularity.
	\end{itemize}\end{theorem}
	
	The subspace $\cU_1$ is studied in \cite{SX3}. In this paper, we only focus on the space $\cU_2$. We would like to highlight that flows in $\cU_2$ may not have nondegenerate singularities; they may have degenerate singularities, but the singular set is not a smooth submanifold.

	It would also be interesting to know the relation between $\cU_1$ and $\cU_2$. The peanut example by Altschuler-Angenent-Giga \cite{AAG} suggests that there exist hypersurfaces that belong to both closures $\overline{\cU_1}$ and $\overline{\cU_2}$, showing that $\overline{\cU_1}$ and $\overline{\cU_2}$ are not disjoint.

	\subsection{Outline of the proof}
	
	The methodology used in this work is substantially more challenging and distinct from the one-sided perturbation approach presented in \cite{CM1, CCMS, SX1, SX2}. To illustrate this point, let us consider an example. Suppose $M$ is an $m\times m$ matrix with distinct eigenvalues $|\lambda_1| > |\lambda_2| > \ldots > |\lambda_m|$. For a generic vector $v$, we have $\lim_{n\to\infty} \lambda_1^{-n}({M^n v})$ converges to $v_1$, the eigenvector corresponding to $\lambda_1$. However, to obtain $v_2$, the eigenvector corresponding to $\lambda_2$, we need to consider $\lim_{n\to\infty} \lambda_2^{-n}({M^n v})$, where $v$ is generic in $\R^n/(\R v_1)$. Therefore, obtaining the eigenvector corresponding to an eigenvalue other than the first eigenvalue from a generic vector is much harder.
	
	In the context of generic MCF, we analyze the linearized operator $L_\Sigma=\Delta-\frac{1}{2} x\cdot\nabla+(|A|^2+\frac{1}{2})$ restricted to the shrinker.  When dealing with nonspherical and noncylindrical shrinkers, our approach in \cite{SX1, SX2} involves introducing an initial positive perturbation to increase the difference between the perturbed RMCF and the unperturbed one towards the direction of the first eigenfunction, analogous to finding $v_1$ in the previous paragraph.
	
	However, when dealing with spherical and cylindrical shrinkers, the first several eigenfunctions correspond to translations and dilations, which do not provide new geometric evolution. Instead, new geometric information is hidden in some higher Fourier modes. Hence, the situation is analogous to finding $v_2$ in the previous paragraph, and we need to consider the higher eigenfunctions modulo the conformal linear transformations\footnote{By conformal linear transformations, we mean the group of tranformations generated by rigid transformations in $\R^{n+1}$ (including translations and rotations) and dilations. }. In \cite{SX3}, when dealing with spherical singularities, we deciphered the information of the regularity of the LSF lying in the eigen direction of the first negative eigenvalue. We then introduced a centering map to effectively modulo the translations and dilation and ensure that an initial perturbation eventually grows towards the eigen direction of the first negative eigenvalue.
	
	However, dealing with cylindrical singularities presents much greater complexity. In this case, there exist functions $h_2(y_i)=c_2(y_i^2-2)$ that lie in the kernel of $L_\Sigma$ and play a crucial role in determining the isolatedness of the singularity. Our objective is to guide our initial perturbation towards $h_2$ since it represents the essential Fourier mode. One might initially consider that eliminating the translations and dilations would be sufficient. However, there are also eigenfunctions $\{y_i\}_{i=1}^k$ associated with a positive eigenvalue $\frac{1}{2}$ that cannot be eliminated by a conformal linear transformation in general. This observation is a key point in the proof of $F$-instability of cylinders as shown in \cite{CM1}. 
	
	To address this difficulty, we need to understand the geometric significance of the eigenfunctions $y_i$. The key insight is that these eigenfunctions correspond to translations along the axis direction and can be eliminated by translation if there is a nontrivial $h_2$-component in the normal form (see Remark \ref{RmkTranslation}). Thus, we encounter a situation where the elimination of $y_i$ and the revelation of $h_2$ are intertwined. In the proof, we introduce a dynamic procedure that involves infinitely many steps of conformal linear transformations along the RMCF to eliminate all the exponentially growing modes progressively.

	\subsection*{Organization of the paper}
	Section \ref{S:Preliminaries} consists of some preliminaries; Section \ref{S_stability} discusses the stability of nondegenerate cylindrical singularities; Section \ref{SDense} discusses how to perturb a degenerate cylindrical singularity to a nondegenerate one under the technical assumptions; Section \ref{SApplication} discusses some applications of the perturbation results.

	 This is the second part of ver 1 of https://arxiv.org/abs/2210.00419v1, with more details. Much progress has been made on this problem since 2022, and it is hard for us to keep track of it all; to reflect the original ideas in this work, we decided to keep our original references

	\subsection*{Acknowledgment}
	We would like to thank Professor Bill Minicozzi for the stimulating discussions, and Professor Natasa Sesum for her interests and comments. We thank Zhihan Wang for his valuable suggestions and comments on the early version of this draft. J. X. is supported by NSFC grants (No. 12271285) in China, the New Cornerstone investigator program and the Xiaomi endowed professorship of Tsinghua University.
	
	\subsection*{AI declaration}
	The main ideas, frames, and first version of this work were completed in 2022. Generative AI was only used to check computational correctness and improve the presentation.

	\section{Preliminaries}\label{S:Preliminaries}
	
	\subsection{Eigenvalues and Eigenfunctions of the $L$-operator}\label{SPre}
	In this section we summarize some previously known results on cylindrical singularities. We will be working in the weighted Sobolev space. Given two functions $f,g$ defined on a hypersurface $\Sigma$, we define an inner product
	$
	\langle f,g\rangle_{L^2} =\int_{\Sigma}f(x)g(x)e^{-\frac{|x|^2}{4}}d\cH^n(x).
	$
	Then we define the weighted $L^2$-norm by
	$\|f\|_{L^2}=\langle f,f\rangle_{L^2}^{1/2},$
	and the weighted $L^2(\Sigma)$ space consists of function $f$ with $\|f\|_{L^2}<+\infty$. Similarly, we define the weighted higher Sobolev space $H^{k}$ with
	$
	H^{k}(\Sigma):=\left\{f: \sum_{i=0}^k \|\nabla^i f\|^2_{L^2}<+\infty\right\}.
	$
	Throughout the paper, for simplicity of notation, we use $\|\cdot\|$ to denote the $H^1$-norm unless otherwise mentioned. We also use $d\mu:=e^{-\frac{|x|^2}{4}}d\cH^n(x)$

	The first three eigenvalues and their corresponding eigenfunctions of $-L_{\cC_{n,k}}$ are listed in Table \ref{TableEigen}. We will use $\Pi_{\lambda_j}$ to denote the $L^2$-projection of a function to the eigenspace spanned by eigenfunctions with eigenvalue $\lambda_j$. These eigenfunctions have geometric meanings: 
	\begin{itemize}
		\item Constant $1$ is the mean curvature on the generalized cylinder, representing infinitesimal (spacetime) dilation.
		\item $\theta_i$ and $y_j$ are both infinitesimal translations. Specifically, $\theta_i$'s represent the translations in the directions of the spherical components, and $y_i$'s represent the translations in the axis directions (see Remark \ref{RmkTranslation}).
		\item $\theta_i y_j$ represents the infinitesimal rotation.
		\item $h_2(y_i)=c_2(y_i^2-2)$ is known to be the non-integrable Jacobi field. It represents \emph{non-degenerate neckpinching} in the related direction on the axis. 
		\item $y_iy_j$'s show up if we rotate $h_2(y_i)$'s in the $\R^k$ space. e.g. $\left(\frac{y_i+y_j}{\sqrt2}\right)^2-2=\frac{y_i^2-2}{2}+\frac{y_j^2-2}{2}+y_iy_j$.
	\end{itemize}
	
	The spectrum of $-\Delta_{\bS^{n-k}(\varrho)}$ is $0,1/2,\frac{n-k+1}{n-k},\cdots$, and the eigenfunctions are known to be the restriction of homogeneous harmonic polynomials. The first several eigenfunctions are listed as follows: constant functions for eigenvalue $0$; $\theta_i$, the restriction of linear functions in $\R^{n-k+1}$ to $\bS^{n-k}(\varrho)$, for eigenvalue $1/2$; and $\theta_i^2-\theta^2_j,\cdots$ for eigenvalue $\frac{n-k+1}{n-k}$. The spectrum of $-L_{\R^k}$ on $\R^{k}$ is given by half integers $\frac{m}{2}-1,\ m=0,1,2,\ldots$, and the eigenfunctions for eigenvalue $\frac{m}{2}-1$ are given by $h_{m_1}(y_1)\cdots h_{m_k}(y_k)$ with $m_1+\ldots+m_k=m$, where $h_{m_i}(y_i)=c_{m_i}\tilde h_{m_i}(y_i/2)$, and $\tilde h_{m_i}$ are standard Hermite polynomials with $c_{m_i}=2^{-m_i/2}(4\pi)^{-1/4}(m_i!)^{-1/2}$ are normalizing factors such that $\|h_{m_i}\|_{L^2(\R)}=1$. In particular, we have $$\tilde h_0(x)=1,\ \tilde h_1(x)=2x,\ \tilde h_2(x)=4x^2-2 \mathrm{\ and\ } c_0=(4\pi)^{-1/4}, c_1=4^{-1/2}\pi^{-1/4}, c_2=2^{-2}\pi^{-1/4}.$$ We shall use the following fact (c.f. Appendix B of \cite{HV2}).
	\begin{lemma}\label{LmTriple}
		Let $A_{m,n,\ell}=\int_\R h_m(x)h_n(x)h_\ell(x)e^{-\frac{|x|^2}{4}}dx.$ Then $A_{m,n,\ell}=0$ unless we have $m+n+\ell$ is even and $n\leq m+\ell,\ m\leq n+\ell,\ \ell\leq m+n$, in which case we have $$A_{m,n,\ell}=(4\pi)^{-1/4}(m!n!\ell!)^{1/2}\left(\left(\frac{m+n-\ell}{2}\right)!\left(\frac{n+\ell-m}{2}\right)!\left(\frac{m+\ell-n}{2}\right)!\right)^{-1}.$$ In particular, we have $A_{2,2,2}=2\pi^{-1/4}=8c_2.$
	\end{lemma}

	\begin{table}[H]
		\begin{tabular}{|l|l|}
			\hline
			eigenvalues of $-L_{\cC_{n,k}}$ & corresponding eigenfunctions \\ \hline
			$-1$ & $1$ \\ \hline
			$-1/2$ & $\theta_i,y_j,\ i=1,2,\ldots,n-k+1,\ j=1,2,\ldots,k$ \\ \hline
			$0$ & $\theta_iy_j,\ h_2(y_j)=c_2(y_j^2-2),\ y_{j_1}y_{j_2}$ \\ \hline
			$\min\{1/(n-k),1/2\}$ & $\ldots$ 
			\\ \hline
		\end{tabular}
		\caption{Eigenvalues and eigenfunctions of $-L_{\cC_{n,k}}$.}
		\label{TableEigen}
	\end{table}
	
	Note that while $h_{m_i}(y)$ is a normalized eigenfunction in $\|\cdot\|_{L^2(\R)}$, $h_{m_i}(y_i)$ is not a normalized eigenfunction in $\|\cdot\|_{L^2(\cC_{n,k})}$. Thus, we define $H_{m_1,m_2,\cdots,m_k}$ be a multiple of $h_{m_1}h_{m_2}\cdots h_{m_k}$, such that $H_{m_1,m_2,\cdots,m_k}$ is a normalized eigenfunction in $\|\cdot\|_{L^2(\cC_{n,k})}$. In particular, $H_2(y_i)=c_0^{k-1}\cG_{n,k}^{-1/2}h_2(y_i)$, $H_{1,1}(y_i,y_j)=c_0^{k-2}\cG_{n,k}^{-1/2}h_1(y_i)h_1(y_j)$. Here $$\cG_{n,k}=\int_{\mathbb S^{n-k}(\varrho)}e^{-\frac{|x|^2}{4}}d\cH^{n-k}(x)=\varrho^{n-k}e^{-\frac{\varrho^2}{4}}\omega_{n-k},$$ where $\omega_{n-k}$ is the area of the $(n-k)$-dimensional unit sphere.
	
	Among all the eigenfunctions, we would like to draw readers' attention to $y_i$ corresponding to the eigenvalue $1/2$. Perturbing the cylinder in the $y_i$ direction can decrease the value of the $F$-functional, which translations or dilations cannot compensate for. This is the key idea in the proof of $ F$-instability of cylinders in \cite[Section 11]{CM1}. However, one of the key observations in the present paper is that in the presence of a nontrivial $h_2(y_i)=c_2(y_i^2-2)$ component, a perturbation in the $y_i$-direction can be killed. Let us elaborate on this point in the following remark. 
	
	\begin{remark}[Geometric meaning of the eigenfunctions $y_i$]\label{RmkTranslation}
		
		The eigenfunctions $y_j$'s represent the infinitesimal translations in the cylinder axis directions. These translations may not be seen in many cases. For example, if the hypersurface is exactly the generalized cylinder, it is translation invariant in the axis direction, then one can not see the infinitesimal translation represented by $y_i$'s.
		
		However, if the RMCF converges to $\cC_{n,k}$ in a ``pinching" manner, then $y_i$ shows up as an infinitesimal translation. For simplicity, let us suppose $\bM_t$ is an MCF with a spacetime singularity at $(0,0)$, that is modeled by a generalized cylinder $\cC_{n,1}$, where $\R$ has coordinate $y$. We assume the singularity is non-degenerate and in later section, we will prove that the dominant term in the RMCF is $h_2=c_2(y^2-2)$, namely if we write the RMCF $M_\tau$ as the graph of a function $u(\cdot,\tau)$ over $\cC_{n,1}$, in some large ball, then 
		$u(\cdot,\tau)=\frac{a}{\tau}(y^2-2)+o(1/\tau)$
		for some constant $a$. As a consequence, if we take a look at the spacetime blowing up of the MCF at a fixed nearby spacetime point $(\bar{z},0)$ where $\bar{z}=(0,z)\in\R^{n}\times\R$. Then the RMCF for $(\bar{z},0)$ can be express as the graph of 
		\begin{equation}\label{EqTranslation}
			v(\cdot,\tau)=\frac{a}{\tau}((e^{\tau/2}\bar z+y)^2-2)+o(1/\tau)
			=
			\frac{ae^\tau}{\tau}\bar{z}^2+\frac{2a}{\tau}e^{\tau/2}\bar{z} y+\frac{a}{\tau}(y^2-2)+o(1/\tau).
		\end{equation}
		Because $\bar{z}$ is a constant, $\frac{ae^\tau}{\tau}\bar{z}^2$ represent a spacetime dilation term. Then this implies that $y$ represents the translation in the axis direction, after modulo some dilations.
	\end{remark}
	While we will be mainly working with the (weighted) $H^1$-norm, when we project a function to some finite-dimensional eigensubspace, we only need to consider the $L^2$-norm. The following lemma is simply a consequence of integration by parts.
	
	\begin{lemma}
		Suppose $L_{\cC_{n,k}}\phi+\mu\phi=0$. Then for any $u\in H^1(\cC_{n,k})$,
		\[
		\|\Pi_\phi u\|_{H^1}=(2+\mu)^{1/2}\|\Pi_\phi u\|_{L^2}.
		\]
	\end{lemma}
	\begin{proof}
		Using integration by parts,
		\begin{align*}
			\|\phi\|_{H^1}^2
			=&
			\int \phi^2 +|\nabla \phi|^2 d\mu
			=
			\int  \phi(\phi-\cL \phi) d\mu
			=
			\int \phi(1+\mu+1)\phi d\mu
			=(2+\mu)\|\phi\|_{L^2}^2.
		\end{align*}
	\end{proof}

	\subsection{The evolutionary equation and the difference equation}
	To study the asymptotics of the rescaled mean curvature flow approaching the limiting cylinder, we need to study the evolution of the graphical function. Let us briefly recall the analysis of the graphical function of rescaled mean curvature flow in \cite[Appendix A]{SunXue25_GenericI}. Suppose $u(\cdot, \tau):\ \cC_{n,k}\cap B_{\mathbf r(\tau)}\to \R$ is the rescaled mean curvature flow $M_\tau$ restricted to the ball $B_{\mathbf r(\tau)}$, then
	\[
	\pr_\tau u=Lu+\cQ(J^2u).
	\]
	Here $\cQ$ is at least quadratic in $u$ given in \cite[Appendix A]{SunXue25_GenericI} and $J^2u:=(u,Du,D^2u)$ means the $2$-jet of $u$. We shall study the evolution of $u$ under the differential equation. 
	\subsubsection{The cutoff}
	Since $M_{\tau}$ cannot be written as a global graph over $\cC_{n,k}$, we introduce a smooth cutoff function $\tilde \chi(\tau):\ \R_{\geq 0}\to \R$ that is $1$ over $[0,1]$ and vanishes outside $[0,2]$, so that $|\nabla^\ell \tilde \chi|\leq C_{\ell}$ for any $\ell\in\{0,1,2\}$. Then for any $r>1$ we define $\chi_r:\cC_{n,k}\to\R$ by 
	\[
	\chi_r(\theta,y)=
	\begin{cases}
		1,& |y|\leq r,
		\\
		\tilde \chi(|y|-r+1), & |y|\in[r,r+1],
		\\
		0,& |y|\geq r+1.
	\end{cases}
	\]
	In particular, given a differentiable function $f:\R\to\R_{\geq 1}$, we have $\pr_{\tau} \chi_{f(\tau)}(\theta,y)$ is only nonzero in $[f(\tau),f(\tau)+1]$, and $|\pr_{\tau} \chi_{f(\tau)}(\theta,y)|\leq C_1|f'(\tau)|$. In particular, when $f(\tau)=C(\log \tau)^{\alpha}$ or $\tau^\kappa$ for $\alpha>0$ and $\kappa<1$, we have $|\pr_{\tau} \chi_{f(\tau)}(\theta,y)|\to 0$ as $t\to\infty$.
	
	Given a function $\mathbf{r}:\R\to\R_{\geq 0}$, we define $\mathbb A_{\mathbf r(\tau)}$, or simply $\A_{\tau}$ if $\mathbf{r}$ is the graphical function and is fixed, to be the annulus region $\cC_{n,k}\cap(B_{\br(\tau)}\backslash B_{\br(\tau)-1})$. Then we derive the equation for $\chi u$:
	\begin{equation}\label{EqCutOff}
		\pr_{\tau} (\chi u)
		=
		L(\chi u)+\chi\cQ(J^2u)+(u(-\Delta\chi+\pr_{\tau}\chi)-2\langle \nabla\chi,\nabla u\rangle +\frac{1}{2}\langle x,\nabla \chi\rangle u),
	\end{equation}
	where the last term on the RHS is supported on the annulus $\mathbb A_{\mathbf r(\tau)}$. For simplicity, we write
	\begin{equation}\label{Eq}
		\pr_{\tau} (\chi u)
		=
		L(\chi u)+\cB(J^2u).
	\end{equation}
	
	\subsubsection{Modulo the rotations}\label{SSSBrendle}
	The natural function space to study these equations is $H^1(\cC_{n,k})$, which admits a natural direct sum decomposition into eigenspaces of the $L$-operator. Lying in the kernel of $L$, there are eigenfunctions generating rotations, i.e., Fourier modes in span$\{\theta_iy_j\}$, which are not important in our analysis of isolatedness of singularities, thus we choose to modulo them following Brendle-Choi \cite{BC1}. We consider the rotated rescaled mean curvature flow $\tilde M_{\tau}=G_{\tau} M_{\tau},\ G_{\tau}\in \mathrm{SO}(n+1)$, and write $\tilde M_{\tau}$ over $\mathcal{C}_{n,k}\cap B_{\br(\tau)}$ as the graph of a function $\tilde u$, where the rotation $G_{\tau}$ is chosen to make sure that $\chi_{\br(\tau)} \tilde u$ does not have Fourier modes corresponding to rotations. Explicitly, we require the equation $\int \tilde u\chi_{\br(\tau)} \langle Ax,\mathbf{n}(x)\rangle e^{-\frac{|x|^2}{4}}=0$ holds for all $A\in \mathfrak{so}(n+1)$ at every moment $t,$ where $\mathbf n(x)$ is the unit outer normal of the cylinder $\mathcal C_{n,k}$ at the point $x.$ We have that $\tilde u$ solves the equation $\partial_{\tau} \tilde u=L\tilde u+\mathcal Q(\tilde u)+\mathcal A$, where $\mathcal A=\langle\dot{G}_{\tau} G_{\tau}^{-1}x,\mathbf n(x)\rangle$ (see \cite{BC1}, Lemma 2.4) and the evolution of $\overline{v}=\tilde u\chi_{\br(\tau)}$ with cutoff satisfies the equation
	\begin{equation}\label{eq:bar equation}
		\pr_{\tau} \overline{v}=L\overline{v}+\overline{\cB}(J^2\overline{v})+\chi_{\br(\tau)}\mathcal A,
	\end{equation}
	where $\overline{\cB}(J^2\overline{v})=\cB(J^2\overline{v})+O(|G(\tau)||\bar v|)$ (see \cite[Lemma 2.4, 2.5]{BC1}), and  the RHS is zero when projected to the eigenmodes corresponding to rotations, since $\bar v$ does not have those Fourier modes by the choice of $G_{\tau}$.

	We have the following lemma (Lemma 5.1 of \cite{SunXue25_GenericI}).
	\begin{lemma}\label{LmA} 
		$	\|\chi_{\br(\tau)}\cA\|_{L^2}\leq 2\|\overline{\mathcal B}\|_{L^2}$.
	\end{lemma}

	\subsubsection{The difference equation}
	\label{SSDifferenceEquation}
	
	Let $u_1$ and $u_2$ be graphical rescaled mean curvature flows over $\cC_{n,k}\cap B_{\br(\tau)}$ and assume $\|u_i(\cdot,\tau)\|_{C^{2,\alpha}(B_{\br(\tau)})}
	\leq
	\eps_0.$ Define
	\[
	w=u_1-u_2,
	\qquad
	v=\chi_{\br(\tau)}w.
	\]
	Subtracting the two cutoff graph equations gives
	\begin{equation}
		\label{EqDifferenceCutoff}
		\partial_\tau v
		=
		Lv+\delta\cB.
	\end{equation}
	After replacing $u_i$ by $\chi_{\br(\tau)}u_i$ in the nonlinear terms, we may write
	\begin{equation}
		\label{EqDifferenceLinearized}
		\delta\cB
		=
		P(\tau)v+\cE,
	\end{equation}
	where
	\[
	P(\tau)v
	=
	\int_0^1
	D\cB
	\left(
	J^2(\chi_{\br(\tau)}u_2)+sJ^2v
	\right)
	[J^2v]
	\,ds,
	\]
	and $\cE$ is supported in the annulus
	\[
	\mathbb A_{\br(\tau)}
	=
	\cC_{n,k}
	\cap
	\left(
	B_{\br(\tau)}
	\setminus
	B_{\br(\tau)-1}
	\right).
	\]
	
	The coefficients of $P(\tau)$ satisfy
	\[
	\|P(\tau)v\|_{L^2}
	\leq
	C\eps_0\|v\|_{H^2},
	\]
	and
	\[
	\|\cE\|_{L^2}
	\leq
	C
	e^{-\br(\tau)^2/8}
	\left(
	\|u_1\|_{C^2(\mathbb A_{\br(\tau)})}
	+
	\|u_2\|_{C^2(\mathbb A_{\br(\tau)})}
	\right).
	\]

	\subsection{The Hilbert space and cones}

	To fix a notation, throughout the paper we use $\eps_0$ to be a small number giving upper bounds for the $\|u\|_{C^2}$ for the graphical function of the RMCF $\cM_{\tau}$ when restricted in the ball $B_{\mathbf r(\tau)}$.
	
	The following result was proved in \cite{SunXue25_GenericI} showing that the nonlinear equation can be well approximated by the linear one.
	\begin{proposition}[Additive one-step approximation; Proposition 5.2 of
		\cite{SunXue25_GenericI}] \label{PropApproximation}
		Let
		$v(\tau)=\chi_{\br(\tau)}(u_1(\tau)-u_2(\tau))$ satisfy
		\eqref{EqDifferenceCutoff}--\eqref{EqDifferenceLinearized}, put
		\[
		\omega(\tau)=\exp\!\left(-\frac{\br(\tau)^2}{5}\right),
		\qquad
		\eps_m=
		\sup_{\tau\in[m,m+1]}
		\max_{i=1,2}\|u_i(\tau)\|_{C^2(B_{\br(\tau)})}.
		\]
		For every sufficiently large integer $m$,
		\begin{equation}\label{EqOneStepAdditive}
			\|v(m+1)-e^Lv(m)\|
			\leq C\eps_m\bigl(\|v(m)\|+\omega(m)^{1/2}\bigr).
		\end{equation}
		Consequently the multiplicative estimate
		$\|v(m+1)-e^Lv(m)\|\leq2C\eps_m\|v(m)\|$
		is available only on the active scale
		$\|v(m)\|\geq\omega(m)^{1/2}$.
	\end{proposition}

	To continue, it is convenient to introduce the splitting $H^1=E^+\oplus E^0\oplus E^-$, where $E^+$ is spanned by Fourier modes with positive eigenvalues of $L$. Similarly for $E^0$ and $E^-$. As the linear equation is contracting in $E^-$ and expanding in $E^+$, it is also convenient to introduce the following cones. Let $\alpha>0$ be a positive number. We introduce two cones $\cK_{\geq 0}$ and $\cK_{0}$ as follows
	$$\cK_{\geq 0}(\alpha):=\left\{u=(u_+,u_0,u_-)\in E^+\oplus E^0\oplus E^-\ |\ \|u_++u_0\|\geq \alpha\|u_-\|\right\}$$
	is a $\alpha$-cone around $E^+\oplus E^0$ and 
	$$\cK_{0}(\alpha):=\left\{u=(u_+,u_0,u_-)\in E^+\oplus E^0\oplus E^-\ |\ \|u_0\|\geq 
	\alpha\|u_++u_-\|\right\}$$
	is a $\alpha$-cone around $ E^0$. Both are narrow when $\alpha$ is large. We also set
	\[
	\mathcal K_{>0}(\alpha)
	:=\{u=u_++u_0+u_-:\ \|u_+\|\geq
	\alpha\|u_0+u_-\|\},
	\]
	the cone around the strictly positive spectral subspace.
	
	The following theorem was proved in \cite{SunXue25_GenericI}.
	\begin{theorem}[Cone theorem; Proposition 5.3 of
		\cite{SunXue25_GenericI}]
		\label{ThmCone}
		Suppose $v$ satisfies Proposition \ref{PropApproximation}.  There are $R_*>1$, $\alpha_*>0$, $\nu>0$, and $q\in(0,1)$ such that, whenever
		\begin{equation}\label{EqActiveBoundaryScale}
			\|v(m)\|\geq R_*\omega(m)^{1/2},
		\end{equation}
		the following statements hold for all sufficiently large $m$:
		\begin{enumerate}
			\item if $v(m)\in\cK_{\geq0}(\alpha)$ and $C\eps_m\leq\alpha\leq c\eps_m^{-1}$, then $v(m+1)\in\cK_{\geq0}((1+\nu)\alpha)$; 
			\item if $v(m)\in\cK_{>0}(\beta)$ and $C\eps_m\leq\beta\leq c\eps_m^{-1}$, then $v(m+1)\in\cK_{>0}((1+\nu)\beta)$ and
			\[
			\|\Pi_+v(m+1)\|\geq(1+\nu)\|\Pi_+v(m)\|;
			\]
			\item if $v(m)\notin\cK_{\geq0}(\alpha_*)$, then
			\begin{equation}\label{EqNormalizedStableContraction}
				\frac{\|v(m+1)\|}{\omega(m+1)^{1/2}}
				\leq q\frac{\|v(m)\|}{\omega(m)^{1/2}}+C\eps_m.
			\end{equation}
		\end{enumerate}
	\end{theorem}
	
	\subsection{Higher estimate and improvement of error estimate}
	A central technique for the regularity theory in \cite{SunXue25_GenericI} is a generalized version of the Ornstein-Uhlenbeck regularization, originally from \cite{V1}. It allows one to obtain a large range $C^1$-estimate based on the $L^2$-estimate of the solution to the linearized equation with a nicely controlled error.

	\begin{proposition}[Ornstein-Uhlenbeck regularization, Proposition 3.1 in \cite{SunXue25_GenericI}]\label{PropVelazquez} 
		
		Let $0<\xi<2$, $\eta>0$, and $C_0,C_F>0$.  There are $\eps_*>0$, $c_*=c_*(\kappa)>0$, and $T_0>0$ such that the following holds.  Suppose $0\leq\eps_0\leq \min\{\eps_*,c_*\eta\}$, and $Z\geq0$ satisfies the following  on the interval   $[s_\tau,\tau]\subset[T_0,\infty)$: 
		
		\begin{equation}\label{EqVelazInequality}
			\partial_rZ-LZ
			\leq\eps_0Z+C_F\left(
			\frac{1+|y|^2}{r^2}
			+|y|\check\chi(y,r)\right),
		\end{equation}
		where \begin{equation}\label{eq:OU-cutoff-buffer}
			0\leq\check\chi\leq1,\qquad
			\check\chi(y,r)=0\quad\hbox{when }|y|\leq2R(r).
		\end{equation}
		Assume
		\begin{equation}\label{eq:OU-L2-assumption}
			\|Z(\cdot,r)\|_{L^2}
			\leq C_0r^{-\xi},
			\qquad s_\tau\leq r\leq\tau.
		\end{equation}
		Then, for every sufficiently large $\tau$,
		\begin{equation}\label{eq:OU-loss-estimate}
			\begin{split}
				\|Z(\cdot,\tau)\|_{C^0(B_{R(\tau)})}
				\leq C_\eta C_0\tau^{2\kappa-\xi+\eta}.
			\end{split}
		\end{equation}
	\end{proposition}
	In the following applications, we consider $Z$ to be $|\chi w|+|\nabla(\chi w)|$, where $w$ is the solution to the RMCF equation over $\cC_{n,k}$ in a large ball of radius $R$, and $\chi$ is a smooth cut-off function.

	\section{Stability of nondegenerate singularities}\label{S_stability}
	In this section, we prove the stability of nondegenerate singularities. In other words, suppose the MCF $\bM_t,\ t\in[-1,0)$ has a nondegenerate cylindrical singularity at the spacetime $(0,0)$, we will show that any sufficiently small perturbation of $\bM_{-1}$ also admits a nondegenerate cylindrical singularity.

	Let $M_\tau=e^{\tau/2}\bM_{-e^{-\tau}}$, $\tau\in[0,\infty),$ be the corresponding rescaled mean curvature flow.  Fix $\vartheta\in(4/5,1)$ and choose
	\[
	1-\vartheta<\kappa<\min\{\vartheta/4,1/4\}.
	\]
	By the normal form theorems in \cite{SunXue25_GenericI}, after choosing the axial coordinates, using the gauge of Section \ref{SSSBrendle} and increasing $T$ if necessary, $M_\tau$ is a graph of a function $u(\theta,y,\tau)$ over
	$\cC_{n,k}\cap B_{K\sqrt{\tau}}$ and
	\begin{equation}\label{EqNF}
		\left\|
		u(\cdot,\tau)
		-
		\frac{\varrho}{4\tau}
		\sum_{i=1}^k(y_i^2-2)
		\right\|_{H^1(\cC_{n,k}\cap B_{K\sqrt{\tau}})}
		\leq
		C\tau^{-1-\vartheta}.
	\end{equation}
	
	We shall apply the centering argument on the smaller graphical scale
	\[
	R(\tau)=\tau^\kappa.
	\]
	It follows from \eqref{EqNF} that
	\begin{equation}\label{EqNF-small-scale}
		\left\|
		\chi_{R(\tau)}u(\cdot,\tau)
		-
		\frac{\varrho}{4\tau}
		\sum_{i=1}^k(y_i^2-2)
		\right\|
		\leq
		C\tau^{-1-\vartheta}.
	\end{equation}
	In particular, for all sufficiently large $\tau$,
	\begin{equation}\label{Eq-size-neutral}
		\frac{c}{\tau}
		\leq
		\|\chi_{R(\tau)}u(\cdot,\tau)\|
		\leq
		\frac{C}{\tau},
		\qquad
		\|\PiNZ(\chi_{R(\tau)}u)\|
		\leq
		C\tau^{-\vartheta}
		\|\PiZ(\chi_{R(\tau)}u)\|.
	\end{equation}
	Thus,
	\[
	\chi_{R(\tau)}u(\cdot,\tau)
	\in
	\cK_{\geq 0}(c\tau^\vartheta).
	\]
	
	Let $\rmM(\tau)$ be the symmetric matrix of the quadratic Hermite coefficients, with the same normalization as in Section 5.5 of \cite{SunXue25_GenericI}. As long as the graphical function remains in the above neutral cone, the calculation in
	\cite{SunXue25_GenericI} gives
	\begin{align}\label{EqM}
		|\rmM'+\gamma\rmM^2|
		&\leq C(1+R(\tau)^2)
		\left(\alpha(\tau)^{-1}+\bar\eps(\tau)\right)|\rmM|^2
		+C\bar\eps(\tau)e^{-R(\tau)^2/5} \notag\\
		&\leq C\tau^{-\sigma}|\rmM|^2
	\end{align}
	
	for some $\sigma\in(0,1)$.  Here $\gamma>0$ is the dimensional constant from \cite{SunXue25_GenericI}, $\alpha(\tau)=c\tau^{\vartheta}$, and $\bar\eps(\tau)$ is the graphical $C^2$-norm on $B_{R(\tau)}$. Indeed, the small-radius normal form and interior parabolic estimates give $\bar\eps(\tau)\leq C\tau^{2\kappa-1}$; hence the two polynomial losses are $O(\tau^{2\kappa-\vartheta})$ and $O(\tau^{4\kappa-1})$.  Both decay by the choice of $\kappa$, while the Gaussian cutoff term is absorbed using
	\eqref{Eq-size-neutral}.

	We shall use the following consequence of \eqref{EqM}. Recall that by the variational theory of matrix-valued ODE, we can list the eigenvalues of  $\rmM$ as $\lambda_1,\cdots,\lambda_k$, which may not be ordered and can cross each other.

	\begin{lemma}[Matrix Riccati asymptotics; Lemma 5.8 of
		\cite{SunXue25_GenericI}]
		\label{LmNeutralPersistence}
		Let $\rmM:[T,\infty)\to\operatorname{Sym}_k$ satisfy
		$\rmM(\tau)\to0$ and, for some $\sigma\in(0,1)$,
		\[
		|\rmM'+\gamma\rmM^2|
		\leq C\tau^{-\sigma}|\rmM|^2.
		\]
		If $|\rmM(\tau)|\geq c/\tau$ for all sufficiently large $\tau$, then there is a nonzero orthogonal projection $P_\infty$ such that
		\begin{equation}\label{EqMatrixRiccatiAsymptotic}
			\rmM(\tau)=\frac{P_\infty}{\gamma\tau}
			+O(\tau^{-1-\sigma}).
		\end{equation}
		If, in addition, every eigenvalue of $\rmM(\tau)$ is bounded below by $c/\tau$ eventually, then $P_\infty=I_k$ and the normal form is nondegenerate.
	\end{lemma}
	We next consider a perturbation of the initial hypersurface. Fix a sufficiently large time $T$. Since the original flow is smooth on the compact time interval $[-1,-e^{-T}]$, smooth dependence on the initial hypersurface shows that, if the perturbation of $\bM_{-1}$ is sufficiently small in $C^2$, then the perturbed flow exists up to time $-e^{-T}$ and its rescaled time-$T$ slice is arbitrarily close to $M_T$ on every fixed compact set.
	
	Denote the perturbed rescaled flow by $\widetilde M_\tau$ and its graphical function in the gauge of Section \ref{SSSBrendle} by $\widetilde u(\cdot,\tau)$. By first choosing $T$ sufficiently large and then choosing the initial perturbation sufficiently small, we may assume that
	\begin{equation}\label{Eq-perturbed-initial-profile}
		\left\|
		\chi_{R(T)}\widetilde u(\cdot,T)
		-
		\frac{\varrho}{4T}
		\sum_{i=1}^k(y_i^2-2)
		\right\|
		\leq
		CT^{-1-\vartheta},
	\end{equation}
	and that all the graphical, cutoff, and entropy bounds needed below hold.
	\subsection{Centering the flow}\label{SS_centering}
	
	The positive modes $1$, $\theta_\alpha$, and $y_j$ may grow under the rescaled flow. The constant mode and the $\theta_\alpha$-modes can be removed by changing the spacetime scale and center. The $y_j$-modes can also be removed in the presence of a nontrivial quadratic component, as explained in Remark \ref{RmkTranslation}. We therefore introduce a centering map at each sufficiently large integer time.
	
	No additional projection or parameter is needed for the modes already eliminated by the gauge of Section \ref{SSSBrendle}. We use $\Pi_{\psi_\bx}$ and $\Pi_{\psi_\by}$ to denote the projections onto
	\[
	\operatorname{span}\{\theta_\alpha\}
	\quad\text{and}\quad
	\operatorname{span}\{y_j\},
	\]
	respectively. Finally, $\Pi_{-1}$ denotes the projection onto the constant functions. We set
	\[
	\Pi_*
	=
	\Pi_{-1}
	+
	\Pi_{\psi_\bx}
	+
	\Pi_{\psi_\by}.
	\]
	Let $\rmW_*$ be the finite-dimensional parameter space of triples
	\[
	\rmT=(\lambda,\bx,\by),
	\]
	where $\lambda>0$, $(\bx,\by)\in\R^{n-k+1}\times\R^k$. On a fixed rescaled slice, we write
	\[
	e^{\rmT}(M)
	=\lambda M-(\bx,\by).
	\]
	The resulting graph is always understood to be written again in the gauge of Section \ref{SSSBrendle}; no new geometric parameter is introduced by this convention.
	
	\begin{proposition}[Centering map]\label{PropCentering}
		For every $\Lambda>0$, there exist $\eps_0\in(0,1)$, $\xi_0\in(0,1)$, $R_0>0$, $c>0$, and $C>0$ with the following significance.
		
		Suppose $\xi\in(0,\xi_0)$ and $M$ is a hypersurface in $\R^{n+1}$ with entropy bounded by $\Lambda$. Suppose that, for some $R>R_0$, $M\cap B_R$ is the graph of a function $u$ over $\cC_{n,k}\cap B_R$ satisfying
		\begin{enumerate}
			\item $\|u\|_{C^2(B_R)}\leq\eps_0,
			\qquad
			\|\chi_Ru\|
			\geq
			Ce^{-R^2/4}\|u\|_{C^2(B_R)},
			\qquad
			\|\chi_Ru\|\leq\eps_0R^{-2}$;
			
			\item there exists $a>0$ such that
			\[
			\left\|
			\PiZ(\chi_Ru)
			-
			a\sum_{i=1}^k(y_i^2-2)
			\right\|
			\leq
			\frac{\xi}{10}\|\chi_Ru\|;
			\]
			
			\item $\|\PiZ(\chi_Ru)\|
			\geq
			(1-10\xi)\|\chi_Ru\|$;
			
			\item
			\[R\|\Pi_{-1}(\chi_Ru)\|
			+\|\Pi_{\psi_\bx}(\chi_Ru)\|
			+\|\Pi_{\psi_\by}(\chi_Ru)\|
			\leq c\xi\|\chi_Ru\|.
			\]
		\end{enumerate}
		
		Then there exists $\rmT=(\lambda,\bx,\by)\in\rmW_*$ such that
		\begin{equation}\label{Eq-centering-parameter}
			R|\lambda-1|
			+
			|\bx|
			+
			a|\by|
			\leq
			C\xi\|\chi_Ru\|.
		\end{equation}
		Moreover, $|\by|\leq C\xi$, and $e^{\rmT}(M)$ is the graph of a function $v$ over $\cC_{n,k}\cap B_{R-C\xi}$ satisfying
		\begin{enumerate}
			\item
			$
			\Pi_*\bigl(\chi_{R-C\xi}v\bigr)=0;
			$
			
			\item
			$
			\|\PiZ(\chi_{R-C\xi}v)\|
			\geq
			(1-\xi)\|\chi_{R-C\xi}v\|;
			$
			
			\item
			\[
			\left\|
			\PiZ(\chi_{R-C\xi}v)
			-
			a\sum_{i=1}^k(y_i^2-2)
			\right\|
			\leq
			\frac{\xi}{3}\|\chi_{R-C\xi}v\|.
			\]
		\end{enumerate}
	\end{proposition}
	
	\begin{proof}
		By assumptions (2) and (3), after decreasing $\xi_0$ if necessary,
		\begin{equation}\label{Eq-a-comparable-N}
			c\|\chi_Ru\|\leq a\leq C\|\chi_Ru\|.
		\end{equation}
		
		For $\rmT=(\lambda,\bx,\by)$, introduce the weighted parameter norm
		\[
		\|\rmT\|_*^2
		=
		R^2|\lambda-1|^2
		+
		|\bx|^2
		+
		a^2|\by|^2.
		\]
		Consider parameters satisfying $\|\rmT\|_* \leq C_1\xi \|\chi_Ru\|,$ where $C_1$ will be fixed below. By Proposition \ref{prop:AppGraph_TransDila} specialized to translations and dilations and \eqref{Eq-a-comparable-N}, $e^{\rmT}(M)$ is a graph over $\cC_{n,k}\cap B_{R-C\xi}$.  Let $u_{\rmT}$ denote its representative in the gauge of Section \ref{SSSBrendle}.  The gauge correction has no component in the target space at first order, and its remaining contribution is included in $\cE(\rmT)$ below.
		
		We identify each of the three geometric eigenspaces with its coefficient vector in a fixed Gaussian orthonormal basis. In the constant component, we multiply the coefficient vector by $R$, and in the $\theta_\alpha$-component we reverse the sign. With this identification, define
		\[
		P(\rmT)
		=
		\Pi_*\bigl(\chi_{R-C\xi}u_{\rmT}\bigr).
		\]
		The transformation formulas in Corollary \ref{cor_change_TDR} imply  
		\begin{equation}\label{Eq-centering-linearization}
			P(\rmT)
			=
			P(0)
			+
			D\bigl(R(\lambda-1),\bx,a\by\bigr)
			+
			\cE(\rmT),
		\end{equation}
		where $D$ is a fixed positive definite linear map. The positivity in the axial translation component follows from the term $2a\,\by\cdot y$ generated by translating $a\sum_i(y_i^2-2)$.
		
		The same transformation formulas, the cutoff estimate, and
		assumption (1) give
		\begin{equation}\label{Eq-centering-error}
			|\cE(\rmT)|
			\leq
			C\left(\eps_0+\xi+R^{-1}\right)\|\rmT\|_*
			+
			Ce^{-cR^2}\|u\|_{C^2(B_R)}.
		\end{equation}
		Assumption (4) gives
		\[
		|P(0)|\leq c\xi \|\chi_Ru\|.
		\]
		Thus, if $C_1$ is sufficiently large and consequently $c$, $\eps_0$, and $\xi_0$ are sufficiently small and $R_0$ is sufficiently large, we have
		\[
		\left\langle
		P(\rmT),
		\bigl(R(\lambda-1),\bx,a\by\bigr)
		\right\rangle
		>0
		\]
		whenever $\|\rmT\|_*=C_1\xi \|\chi_Ru\|.$
		
		The degree of $P$ on this ball is therefore equal to the degree of the positive definite linear map $D$, which is one. Hence there exists $\rmT$ in the ball such that $P(\rmT)=0.$ This proves the first conclusion and \eqref{Eq-centering-parameter}. In view of \eqref{Eq-a-comparable-N}, the axial estimate also gives $|\by|\leq C\xi.$
		
		Finally, Corollary \ref{cor_change_TDR}, assumptions (2)--(3), and
		\eqref{Eq-centering-error} show that
		\[
		\|\PiZ(\chi_{R-C\xi}u_{\rmT})\|
		\geq
		(1-\xi)\|\chi_{R-C\xi}u_{\rmT}\|
		\]
		and
		\[
		\left\|
		\PiZ(\chi_{R-C\xi}u_{\rmT})
		-
		a\sum_{i=1}^k(y_i^2-2)
		\right\|
		\leq
		\frac{\xi}{3}\|\chi_{R-C\xi}u_{\rmT}\|.
		\]
		Taking $v=u_{\rmT}$ completes the proof.
	\end{proof}

	\subsection{The stability theorem}
	
	Now we are ready to prove the stability of nondegenerate singularities. We will prove a quantitative version of the stability theorem, and the stability of nondegenerate singularities is a direct corollary of this theorem.
	
	\begin{theorem}\label{thm_existence_nondeg_with_ansatz}
		For $\Lambda>0$, there exist $T_0>0$, $C>0$, $\delta>0$ with the following significance: suppose $M$ is a hypersurface with $|A|$ uniformly bounded by $\Lambda$ in $\R^{n+1}$, with entropy bounded by $\Lambda$, and there exist $T>T_0$ such that $M\cap B_{\delta T^{1/2}}$ is a graph of function $u(\theta,y)$ over $\cC_{n,k}$ with
		\begin{equation}
			\left\|u(\theta,y)-\left(\varrho\sqrt{1+\frac{\sum_{i=1}^k(y_i^2-2)}{2T}}-\varrho\right)\right\|_{C^1(B_{\delta T^{1/2}})}\leq CT^{-\vartheta}.
		\end{equation}
		
		If, moreover, $u$ satisfies the assumptions in Proposition \ref{PropCentering}, then after a translation and dilation, the RMCF starting from $M$ converges to $\cC_{n,k}$ in $C_{\loc}^\infty$, and the corresponding singularity is nondegenerate.
	\end{theorem}

	\begin{proof}
		If the assumptions in Proposition \ref{PropCentering} are satisfied, and by Proposition \ref{PropCentering} with 
		\begin{itemize}
			\item $\|\chi u\|=O(T^{-1})$,
			\item $R=\delta T^{\kappa}$,
			\item $a=\frac{\varrho}{4T}$,
			\item $\xi=T^{-\vartheta}$,
		\end{itemize}
		we can find a dilation and translation $\rmT_T=(\lambda_T,\bx_T,\by_T)$ so that $\rmT_T(M)$ is the graph
		of a function $v$ over
		$\cC_{n,k}\cap B_{\delta T^{\kappa}-CT^{-\vartheta}}$, where
		\[
		T^\kappa|\lambda_T-1|+|\bx_T|+T^{-1}|\by_T|
		\leq CT^{-1-\vartheta}.
		\]
		
		Set $R_T=\delta T^{\kappa}-CT^{-\vartheta}$ and let $\chi_{R_T}$ be the corresponding cutoff. Let $\alpha_j=cj^{\vartheta}$ at each integer time $j\geq T$. Then $\chi_{R_T}v(\cdot,T)\in\cK_0(\alpha_T)$.
		
		Moreover, we apply the Proposition \ref{PropApproximation} (with $u_2=0,v=\chi u_1$). The assumptions are verified as follows. For $\tau\in [T,T+1]$, the $C^2$ norm $\|v\|_{C^2(\mathbb A_{R(\tau)})}\leq \eps_0$ is given by a pseudolocality argument (see the following Theorem \ref{thm:pseudo-locality}). The $C^2$ norm $\|v\|_{C^2(B_{R(\tau)})}\leq \eps_0$ is given by Proposition \ref{PropVelazquez}. Finally, the inequality $\eps_0\|\chi_{R(\tau)} v(\cdot, \tau)\|_{H^1}\geq \|v\|_{C^2(\mathbb A_{R(\tau)})} e^{-R(\tau)^2/4}$ is automatic considering $\|\chi v\|_{H^1}$ is of order $1/\tau$ and the choice $R(\tau)=\tau^\kappa.$
		
		Proposition \ref{PropApproximation} implies that the ``$-$" Fourier modes contract and the ``$+$" Fourier modes expand. At time $j+1$, if the graph does not already lie in $\cK_0(\alpha_{j+1})$, Proposition \ref{PropCentering}, with $\xi=(j+1)^{-\vartheta}$, removes the positive modes and restores this cone condition.  Pseudolocality and Proposition \ref{PropVelazquez} provide the same graphical and $C^2$ estimates on the next unit interval, so the construction iterates for every $j\geq T$.  At the $j$-th step the loss of graphical radius is at most $Cj^{-\vartheta}$; hence
		\[
		\delta j^\kappa-C\sum_{m=T}^{j}m^{-\vartheta}
		\geq \frac{\delta}{2}j^\kappa
		\]
		after increasing $T$, because $1-\vartheta<\kappa$.
		
		Note that if we rescale back to the mean curvature flow, $\rmT_j$ corresponds to a conformal transformation $\rmT_j'=(\lambda_j,\bx_j,\by_j)$, where $$|\lambda_j-1|\leq e^{-j/2}j^{-1-\vartheta/2},\ |\bx_j|\leq e^{-j/2}j^{-1-\vartheta},\ \mathrm{and}\ |\by_j|\leq e^{-j/2}j^{-\vartheta}.$$ Therefore, all the compositions of $\rmT_j'$ converge to a limit $\rmT_\infty'$. 
		
		Moreover, the tail of the unrescaled centering maps after time $N$ is $O(e^{-N/2}N^{-\vartheta})$.  Corollary \ref{cor_change_TDR} therefore shows that applying the limiting tail changes the projections at time $N$ by $o(N^{-1-\vartheta})$.  Consequently, the limiting centered flow still lies in $\cK_0(cN^{\vartheta})$ for all sufficiently large $N$.
		
		Then we have the  mean curvature flow of $e^{\rmT_\infty'}(\bM)$  converges to a cylindrical singularity, and  the preceding summation shows that $M'_\tau$ has graphical radius at least $c_0\delta\tau^\kappa$.
		
		Moreover, the neutral modes domination condition enables us to derive the ODE \eqref{EqM} of the neutral modes of the RMCF equation corresponding to $e^{\rmT_\infty'}(\bM)$. The solutions of \eqref{EqM} was analyzed in Section 5 of \cite{SunXue25_GenericI}. In particular, the assumption of Lemma \ref{LmNeutralPersistence} is satisfied (see Lemma 5.5 and 5.6 of \cite{SunXue25_GenericI}) and Lemma \ref{LmNeutralPersistence} implies that the corresponding singularity of $e^{\rmT_\infty'}(\bM)$ is nondegenerate. 
		
	\end{proof}
	
	The following pseudolocality result was proved in \cite{SunXue25_GenericI}.
	
	\begin{theorem}\label{thm:pseudo-locality}
		For any $\eps\in(0,1)$, $\rho\in(0,1)$ and $\sigma>0$, there exist $R_0>0$ and $\eta_0>0$ with the following significance: suppose $M_\tau$ is a rescaled mean curvature flow and $M_{\tau_0}$ is the graph of a function $u(\cdot,\tau_0)$ over $\cC_{n,k}\cap B_R(p)$ for some $R>R_0$, $p\in\{0\}\times\R^k$, and $\|u(\cdot,\tau_0)\|_{C^1(B_R(p))}\leq\eta\leq \eta_0$. Then there exists $\bar\delta >0$ such that when $\tau\in[\tau_0+\sigma,\tau_0+\bar\delta]$, $M_\tau$ is the graph of a function $u(\cdot,\tau)$ over $\cC_{n,k}\cap B_{e^{(\tau-\tau_0)/2}\rho R}(e^{(\tau-\tau_0)/2}p)$ with $$\|\nabla^ku(\cdot,\tau)\|_{L^\infty(\cC_{n,k}\cap B_{e^{(\tau-\tau_0)/2}\rho R}(e^{(\tau-\tau_0)/2}p))}\leq e^{-(k-1)(\tau-\tau_0)/2}\eps, \quad k=0,1,2,3.$$ Moreover, $\bar\delta\to\infty$ as $\eta\to 0$.
	\end{theorem}
	
	We next prove the stability theorems stated in the introduction. 
	
	\begin{proof}[Proof of stability of nondegenerate singularities]
		By \cite{SunXue25_GenericI}, we first fix a sufficiently large time $T$, so that $M_T$ is a graph of a function $u(\cdot,T)$ in $\cC_{n,k}\cap B_{\delta T^{1/2}}$ with 
		\[
		\left\|u(\cdot,T)-\left(\varrho\sqrt{1+\frac{\sum_{i=1}^k(y_i^2-2)}{2T}}-\varrho\right)\right\|_{C^1(B_{\delta T^{1/2}})}\leq CT^{-\vartheta}.
		\]
		Now we consider the perturbed flow to be $\widetilde M_{\tau}$. For any $T\gg 1$, if the perturbation is sufficiently small, the assumptions in Proposition \ref{PropCentering} are satisfied. Then the previous Theorem proves the desired stability.
	\end{proof}

	Finally, we have the following quantitative stability result. 
	\begin{proof}[Proof of Theorem \ref{ThmQuantativeStability}]
		We choose $T$ so large that the normal form holds for $u_{\tau},$ $\tau\geq T.$
		With the assumption of the corollary, we can verify the assumption of Proposition \ref{PropCentering}. Then the argument follows from that of the stability theorem. 
	\end{proof}
	
	As an application, we prove the existence of the rotationally symmetric model for the nondegenerate cylindrical singularities. Given $1\leq k< n$, we prove the existence of an $O(n-k+1)\times O(k)$-invariant mean curvature flow with a nondegenerate cylindrical singularity modeled by $\cC_{n,k}$. When $k=1$, the existence of such a model flow was sketched in the last section of \cite{AV}.

	\begin{proof}[Proof of Theorem \ref{ThmModel}]
		We consider a hypersurface that is a global graph over $\cC_{n,k}$, with the graphical function $u(\theta,y)$ as follows: $u(\theta,y)=f(|y|)$, where $f$ is a single variable smooth function, and when $|z|\leq \delta T^{1/2}$
		\begin{equation}
			f(z)=\left(\varrho\sqrt{1+\frac{z^2-2k}{2T}}-\varrho\right),
		\end{equation}
		and when $|z|\geq 2\delta T^{1/2}$, $f$ is the constant function
		\begin{equation}
			f(z)=\varrho\left(\sqrt{1+\delta^2}-1\right).
		\end{equation}
		When $\delta T^{1/2}\leq |z|\leq 2\delta T^{1/2}$, $f$ is smooth and monotone increasing. Then we apply Theorem \ref{thm_existence_nondeg_with_ansatz} when $T$ is sufficiently large to get a desired RMCF. Note that by the $O(n-k+1)\times O(k)$-invariance of $M$, the Fourier modes $\Pi_{\psi_\bx}\chi_R u$, $\Pi_{\psi_\by}\chi_R u$ all vanish. Thus, during the iteration process, there is no translation or dilation, and the fixed gauge of Section \ref{SSSBrendle} preserves the symmetry. Thus, the RMCF we constructed must also be $O(n-k+1)\times O(k)$-invariant.
	\end{proof}

	\section{Denseness of nondegenerate singularities}\label{SDense}
	In this section, we prove Theorem \ref{ThmGenericity}. The proof is divided into several steps: in Section \ref{SS_difference}, we study the evolution equation of the difference between two RMCFs; particularly, in Section \ref{SS:evolution of neutral}, we study the evolution of Neutral modes. Then we describe how to choose an initial perturbation in Section \ref{SS:ConstructInitialPerturbation}, and how this initial perturbation would imply nondegeneracy after evolving for a definite amount of time in Section \ref{SS_definite time perturbed}. Finally, because of the technical issue brought by the translation and dilation modes, we need to perturb the flow into not a single nearby flow, but a family of flows, and this is discussed in Section \ref{SS:PerturbCondormal}.

	As we have explained in the introduction, the proof differs quite a lot from the perturbation argument as in \cite{SX1,SX2,SX3}, due to the presence of the $y_i$-eigenmodes. The difficulty is overcome by combining the observation in Remark \ref{RmkTranslation} as well as an argument applying a centering map for infinitely many steps along the flow. Moreover, the neutral eigenspace is in general not a one-dimensional space, but essentially $k$-dimensional, spanned by $h_2(y_i)$ for $i=1,2,\dots,k$. Thus, the original RMCF may be degenerate in some directions of $y_i$, and nondegenerate in some other directions $y_j$. 
	
	Let us assume that $M_{\tau}$ is the RMCF, corresponding to an MCF $\bM_t,\ t\in[-1,0)$. Moreover, we assume that
	\begin{itemize}
		\item the normal form of $M_{\tau}$ has leading term $\sum_{i=\ell+1}^k\frac{\varrho}{4\tau}(y_i^2-2)$. In other words, the $y_1,\ldots,y_\ell$-direction are degenerate and the remaining $y_{\ell+1},\ldots, y_k$-direction are nondegenerate. 
	\end{itemize}
	When $\ell=0$, we have that all directions are nondegenerate, and when $\ell=k$, we have that all directions are degenerate. We will show that under the type-I curvature assumption, we can perturb the initial data so that the perturbed flow generates a nondegenerate singularity.
	
	\subsection{Analysis of the difference}\label{SS_difference}
	Let $(M_{\tau})_{\tau\geq 0}$ be an RMCF converging to a cylinder $\mathcal C_{n,k}$ in the $C^\infty_{loc}$ sense as $\tau\to\infty.$ Let $\widetilde{M}_{\tau}$ be another RMCF obtained by slightly perturbing the initial condition. For each moment $\tau$, we write $\widetilde{M}_{\tau}$ as the graph of a function $\tilde w$ over $M_{\tau}$, with graphical radius $\approx O(\tau^{1/2})$. The difference function $\tilde w$ satisfies the following expression.
	
	Let the background manifold $M(\tau)$ in local coordinate have metric $g_{ij}$, second fundamental form $h_{ij}$, and shape operator $h_i^j = g^{jk}h_{ik}$. We further define
	\begin{equation}
		\begin{split}
			\Phi_i^j = \delta_i^j - \tilde w h_i^j,
			&\qquad
			P_i^j = (\Phi^{-1})_i^j,
			\\
			V^i = P_k^i \nabla^k \tilde w,
			&\qquad
			\nu = \sqrt{1 + g_{ij}V^i V^j}.
		\end{split}
	\end{equation}
	Then $\tilde w$ satisfies the equation
	\begin{equation}
		\partial_{\tau} \tilde w = a^{ij} \nabla_i \nabla_j \tilde w + b^k \nabla_k \tilde w + c \tilde w,
	\end{equation}
	where
	\begin{equation}
		a^{ij} = P_k^i P_l^j g^{kl} - \frac{V^i V^j}{\nu^2},
	\end{equation}
	\begin{equation}
		b^m = 2 a^{mi} h_{ij} V^j - \frac{1}{\nu^2} h_{ij} V^i P^{jm} + P_k^m \left[ \tilde w a^{ij} \nabla_i h_j^k - \frac{1}{2} g^{kl} X_l + \tilde w \nabla^k \left( H - \frac{1}{2}\langle X, \nu \rangle \right) \right]    ,
	\end{equation}
	\begin{equation}
		c = P_k^i g^{kj} (h^2)_{ij} + \frac{1}{\nu^2} V^i V^j (h^2)_{ij} + \frac{1}{2}.
	\end{equation}
	Note that when $|\nabla \tilde w|\leq \delta_0$ for some $\delta_0\ll 1$, and $|A|^2\leq C$, we have 
	\begin{equation}
		|b^k \nabla_k \tilde w + c \tilde w|\leq C(1+|X|)(|\tilde w|+|\nabla\tilde w|).
	\end{equation}
	
	The next lemma proves the exponential growth of $\|\tilde w(\cdot,\tau)\|_{C^0}$, provided that it is bounded by a small constant.  
	\begin{lemma}\label{LmCM313} There exist $\dt_1>0$ and $C_1>0$ such that the following holds. Suppose $\tilde w:\ \cup_{\tau\in [\tau_0,\tau_1]} M_{\tau}\to \R$ is a solution to the last equation satisfying $\sup_{\tau\in[\tau_0,\tau_1]}\|\tilde w(\cdot,\tau)\|_{C^2}\leq\dt_1$. Then
		\[
		\|\tilde w(\cdot,\tau)\|_{C^0(M_\tau)}
		\leq e^{C_1(\tau-\tau_0)}
		\|\tilde w(\cdot,\tau_0)\|_{C^0(M_{\tau_0})},
		\qquad \tau\in[\tau_0,\tau_1].
		\]
	\end{lemma}
	\begin{proof}For $\dt_1$ small, $(a^{ij})$ is uniformly elliptic and $c$ is uniformly bounded. At a spatial maximum of $\tilde w$, the gradient term vanishes and $a^{ij}\nabla_i\nabla_j\tilde w\leq0$. Hence
		$\partial_\tau\max\tilde w\leq C_1\max\tilde w$. Applying the same argument to $-\tilde w$ and using the parabolic maximum principle proves the stated $C^0$ estimate.
	\end{proof}
	
	\begin{remark}
		This is the only place where the type-I assumption is used. 
	\end{remark}
	
	We will switch back to the graphs over $\cC_{n,k}$. Now suppose $M(\tau)$ and $\widetilde M(\tau)$ are both graphs over $\cC_{n,k}\cap B_{K\sqrt{\tau}}$ for some $K>K'>0$, on the time interval $[\tau_0,\tau_1]$, whose graphical functions are $u$ and $v$ respectively, with 
	\begin{enumerate}
		\item $\|u\|_{C^0(B_{K'\sqrt{\tau}})}<\eps_0$ and $\|v\|_{C^0(B_{K'\sqrt{\tau}})}<\eps_0$;
		\item on $B_{K\sqrt{\tau}}\backslash B_{K'\sqrt{\tau}}$, $c_0<u<c_1$ and $c_0<v<c_1$; $|\nabla u|\leq \eps_0$, $|\nabla v|\leq\eps_0$,
	\end{enumerate}
	and $w=v-u$. From the proof of Theorem C.2 in \cite{SX1}, we have $\|w(\cdot,\tau)\|_{C^0(B_{K\sqrt{\tau}})}\leq C\|\tilde w(\cdot,\tau)\|_{C^0(M(\tau))}$.
	
	Next, we prove that the $C^2$-norm of $w$ is bounded by the $C^0$-norm in a slightly smaller time interval.
	\begin{lemma}\label{lem_C2 of diff bound by C0}
		In the above setting, for every $K'<K$ with $K'\sqrt{\tau_0}+4n<K\sqrt{\tau_0-1}$, there exists $C_2>0$ only depending on $n$ and $\eps_0$, such that for $\tau\in[\tau_0+1,\tau_1]$,
		\begin{equation}
			\|w(\cdot,\tau)\|_{C^2(B_{K'\sqrt{\tau}})}\leq C_2e^{C_1(\tau-\tau_0)}\|w(\cdot,\tau_0)\|_{C^0(B_{K\sqrt{\tau_0}})}.
		\end{equation}
	\end{lemma}
	
	\begin{proof}
		Fix $\tau\in[\tau_0+1,\tau_1]$ and
		$p\in(\{0\}\times\R^k)\cap B_{K'\sqrt\tau}$.  For
		$s\in[\tau-1,\tau]$ use the moving axial center
		\[
		p(s)=e^{-(\tau-s)/2}p.
		\]
		In coordinates centered at $p(s)$, the velocity of the center cancels
		the constant part of the Ornstein--Uhlenbeck drift.  The coefficients of
		the difference equation are therefore uniformly parabolic with uniformly
		bounded local norms on $B_{4n}(p(s))$.  The choice of $K'$ ensures that
		these balls lie in $B_{K\sqrt s}$.  The interior parabolic Schauder
		estimate gives
		\begin{equation}
			\|w(\cdot,\tau)\|_{C^{2,{\alpha}}(B_{2n}(p))}\leq C_2 \sup_{s\in[\tau-1,\tau]}\|w(\cdot,s)\|_{C^0(B_{4n}(p(s)))}.
		\end{equation}
		Then Lemma \ref{LmCM313} implies the desired bound for $C^2$-norm of $w$.
	\end{proof}
	
	For simplicity, in the following, we shall choose $K$ so large that $K^2\gg C_1$, and $C_1$ slightly larger so that $\|w(\cdot,\tau)\|_{C^2(B_{K\sqrt{\tau}})}\leq e^{C_1(\tau-\tau_0)}\|w(\cdot,\tau_0)\|_{C^0(B_{2K\sqrt{\tau_0}})}$.  We use $K\sqrt\tau$ as the outer graphical and cutoff scale, write $\chi=\chi_{K\sqrt\tau}$, and use the smaller radius $r(\tau)=\tau^\kappa$ in the nonlinear projection estimates. Here $\vartheta>4/5$ and $1-\vartheta<\kappa<\vartheta/4$ are the same proof parameters chosen in Section \ref{S_stability}; the intervening region is controlled by the Gaussian weight.
	
	\subsection{Evolution of the neutral modes}\label{SS:evolution of neutral}
	
	Let $w=v-u$. The normal form, the type-I bound, and interior
	parabolic estimates give the following bounds (so they are consequences
	of the standing hypotheses):
	\begin{enumerate}
		\item $\|u\|_{C^2(B_{2r(\tau)})}+
		\|v\|_{C^2(B_{2r(\tau)})}\leq C\tau^{2\kappa-1}$, while their
		$C^2$ norms on $B_{K\sqrt\tau}$ are bounded by a constant depending
		on $K$; moreover
		$\|u\|_{H^1(B_{K\sqrt{\tau}})}+
		\|v\|_{H^1(B_{K\sqrt{\tau}})}\leq C\tau^{-1}$.
		\item $\vartheta$ is a fixed constant in $(4/5,1)$.
		\item ($w$ has dominance of zero modes) $\|\Pi_{\neq 0}\chi w\|_{H^1}\leq \xi\|\Pi_{=0}\chi w\|_{H^1}$.
	\end{enumerate}
	From Proposition A.1 of the Appendix of \cite{SunXue25_GenericI}, $w$ satisfies the equation
	\begin{equation}
		\begin{split}
			\pr_{\tau} w
			&=
			L_{\Sigma} w
			+\cP_0(v,\nabla v,\nabla^2 v,u,\nabla u,\nabla^2 u)w
			\\&
			\ \ +\cP_1(v,\nabla v,\nabla^2 v,u,\nabla u,\nabla^2 u)\cdot\nabla w+\cP_2(v,\nabla v,\nabla^2 v,u,\nabla u,\nabla^2 u)\cdot\nabla^2 w.
		\end{split}
	\end{equation}
	Note that the quadratic difference is given by
	\begin{equation}\label{eq:quadratic error}
		\begin{split}
			&\left(v^2+4v\Delta_\theta v+2|\nabla_\theta v|^2\right)
			-
			\left(u^2+4u\Delta_\theta u+2|\nabla_\theta u|^2\right) 
			\\
			&=(2u+w+4\Delta_\theta u+4\Delta_\theta w)w+\langle (4\nabla_\theta u+2\nabla_\theta w),\nabla_\theta w\rangle+ 4u\Delta_\theta w
			\\
			&=: \bar \cP_0 w+\bar \cP_1 \cdot\nabla_\theta w+\bar \cP_2 \cdot\nabla_\theta^2 w.   \end{split}
	\end{equation}
	\begin{proposition}[Projection estimate from the Part-I nonlinearity]
		\label{PropNeutralProjection}
		The quadratic term computed in Proposition A.1 of
		\cite{SunXue25_GenericI} is
		\begin{equation}\label{EqQ-PartII}
			\cQ_2(f)=-\frac1{2\varrho}
			\left(f^2+4f\Delta_\theta f+2|\nabla_\theta f|^2\right).
		\end{equation}
		For functions $a,b$, define its polarization by
		\[
		\cB_2(a,b)=-\frac1{2\varrho}
		\left(ab+2a\Delta_\theta b+2b\Delta_\theta a
		+2\langle\nabla_\theta a,\nabla_\theta b\rangle\right).
		\]
		Let $\phi$ be a fixed zero eigenfunction of $L$, let $w_0=\PiZ(\chi w)$, $z=\PiNZ(\chi w)$, and put
		\[
		\epsilon=\max\{\|u\|_{C^2(B_{2r})},\|v\|_{C^2(B_{2r})}\}\leq\eps_0,
		\qquad r=\tau^\kappa,\quad R=K\sqrt\tau.
		\]
		Then
		\begin{align}
			\label{EqCorrectProjection}
			&\left|\frac d{d\tau}\langle\chi w,\phi\rangle
			-\left\langle
			\cB_2\bigl(\chi(u+v),w_0\bigr),\phi
			\right\rangle\right| \notag\\
			&\quad\leq
			C_\phi(1+r^2)\epsilon\|z\|_{H^1}
			+C_\phi\epsilon^2\|\chi w\|_{H^1}
			+C_\phi e^{-cr^2}\|w\|_{C^2(B_R)}
			+|\langle\widetilde\cP,\phi\rangle|.
		\end{align}
		For the cutoff $R=K\sqrt\tau$, the last term is bounded by
		$C(1+K^2\tau)^2e^{-K^2\tau/4}\|w\|_{C^2}$ under the annular graphical
		bound.
	\end{proposition}
	\begin{proof}
		Write $\cQ=\cQ_2+\cC$ as in Proposition A.1 of
		\cite{SunXue25_GenericI} and \eqref{EqQ-PartII}. Exact polarization gives
		\[
		\cQ_2(J^2(\chi v))-\cQ_2(J^2(\chi u))
		=\cB_2\bigl(\chi(u+v),\chi w\bigr).
		\]
		Self-adjointness and $L\phi=0$ remove the linear term. Replacing $\chi w=w_0+z$ by $w_0$ in the last display compensate the first term on the right of \eqref{EqCorrectProjection}: integrate each spherical Hessian falling on $z$ once against the Gaussian measure, so only $\|z\|_{H^1}$ is used. Proposition A.1(3) of \cite{SunXue25_GenericI} applies the same integration by parts to the cubic remainder and gives
		\[
		\left|\langle
		\cC(J^2(\chi v))-\cC(J^2(\chi u)),\phi\rangle\right|
		\leq C_\phi\epsilon^2\|\chi w\|_{H^1}
		+C_\phi e^{-cr^2}\|w\|_{C^2(B_R)}.
		\]
		The last terms in these two estimates bound the region outside $B_r$ using its Gaussian weight. All terms created by commuting the cutoff are supported in the outer annulus and are of the order of $\widetilde\cP$, which proves the assertion.
		
	\end{proof}
	
	For a function $f$, let $\mathsf M(f)$ denote the symmetric matrix of its quadratic Hermite coefficients with exactly the normalization used in Section 5.5 of \cite{SunXue25_GenericI}; explicitly,
	\[
	\mathsf M(f)_{ii}=\sqrt2c_0\langle f,h_2(y_i)\rangle,
	\qquad
	\mathsf M(f)_{ij}=\langle f,h_1(y_i)h_1(y_j)\rangle\quad(i\ne j).
	\]
	The gauge fixed in Section \ref{SSSBrendle} removes the other neutral modes. Let
	\begin{equation}\label{eq_expression of errors in ODE}
		A=\mathsf M(\chi w)=
		\begin{pmatrix}A_1&A_{12}\\ A_{12}^{T}&A_2\end{pmatrix},
		\qquad
		A_1\in\operatorname{Sym}_{\ell},\quad
		A_2\in\operatorname{Sym}_{k-\ell}.
	\end{equation}
	
	Using \eqref{eq_expression of errors in ODE}, we obtain the following block form of the matrix equation.
	\begin{lemma}[Evolution of the neutral coefficients]\label{lem:neutral-block-ode}
		Suppose the hypotheses of Proposition \ref{PropNeutralProjection} hold and the normal form of the unperturbed graph is degenerate in the first $\ell$ directions and nondegenerate in the remaining directions. In the normalization above,
		\[
		\mathsf M(\chi u)=\frac1{\gamma\tau}
		\begin{pmatrix}0&0\\0&I_{k-\ell}\end{pmatrix}+R_u,
		\qquad |R_u|\leq C\tau^{-1-\vartheta},
		\]
		where $\gamma>0$ is the dimensional constant from \cite{SunXue25_GenericI}. If
		\[
		\|\PiNZ(\chi w)\|_{H^1}
		\leq\tau^{-\vartheta}\|\PiZ(\chi w)\|_{H^1},
		\]
		then the coefficient blocks satisfy the equations
		\begin{equation}\label{EqODEDense}
			\begin{aligned}
				\dot A_1
				&=-\gamma\bigl(A_1^2+A_{12}A_{12}^{T}\bigr)
				+\Rem_1+\Err_1,\\
				\dot A_2
				&=-\frac{2}{\tau}A_2
				-\gamma\bigl(A_2^2+A_{12}^{T}A_{12}\bigr)
				+\Rem_2+\Err_2,\\
				\dot A_{12}
				&=-\frac1\tau A_{12}
				-\gamma\bigl(A_1A_{12}+A_{12}A_2\bigr)
				+\Rem_{12}+\Err_{12}.
			\end{aligned}
		\end{equation}
		Here $\Err_\bullet$ denotes the corresponding projected cutoff error. With $|A|=|A_1|+|A_2|+|A_{12}|$ and $\sigma=\min\{\vartheta-4\kappa,1-4\kappa\}>0$,
		\begin{equation}\label{eq:Qbound}
			|\Rem_1|+|\Rem_2|+|\Rem_{12}|
			\leq C\left(\tau^{-\sigma}|A|^2
			+\tau^{-1-\sigma}|A|\right).
		\end{equation}
	\end{lemma}
	
	\begin{proof}
		Proposition \ref{PropNeutralProjection}, the normal-form estimate, and the Hermite product identities in Lemma \ref{LmTriple} give the matrix identity
		\[
		A'=-\gamma\bigl((\mathsf M(\chi u)+A)^2
		-\mathsf M(\chi u)^2\bigr)
		+\Rem+\Err.
		\]
		Substituting $\mathsf M(\chi u)=(\gamma\tau)^{-1} \operatorname{diag}(0,I_{k-\ell})+R_u$ and multiplying the block matrices gives \eqref{EqODEDense}. Terms containing $R_u$ are bounded by $C\tau^{-1-\vartheta}|A|$. On $B_{2r(\tau)}$, the normal form gives $\epsilon=O(\tau^{2\kappa-1})$; the neutral-cone estimate and \eqref{EqCorrectProjection} therefore bound the remaining non-neutral and higher-order contributions by the right-hand side of \eqref{eq:Qbound}. The portion outside $B_{r(\tau)}$ is absorbed by its Gaussian weight. This proves the lemma.
	\end{proof}
	
	\subsection{Constructing the initial perturbation}\label{SS:ConstructInitialPerturbation}
	
	Now we suppose $M(\tau)$ is a RMCF that converges to $\cC_{n,2}$ locally smoothly as $t\to\infty$. Let $\eps_0$ be a small number as in Proposition \ref{PropCentering}, $K\gg0$ is a number to be determined, and suppose $T$ is a large time such that $M_T\cap B_{2K\sqrt T}$ can be written as the graph of $u:\ \cC_{n,k}\cap B_{2K\sqrt T}\to \R$ with that the normal form theorem \ref{ThmNF-sqrtt} applies to $M(\tau)$. Later, we will be working on $B_{K\sqrt T}\subset B_{2K\sqrt T}$, so that we can apply Lemma \ref{lem_C2 of diff bound by C0}.
	
	Let us first recall Lemma 5.2 of \cite{HV1}, and its Sobolev version proved in \cite{SX3}. 
	\begin{lemma}\label{LmHV}Let $M_{\tau}$ be a RMCF and $\mathcal S(t,s)$ be the fundamental solution to the linearized equation $\partial_{\tau} v=L_{M_{\tau}}v$ such that the function $\mathcal S(t,s) v(s)$ solves the equation with initial condition $v(s):\ M_s\to \R$. Then the operator $\mathcal S(\tau,0):\ L^2(M_0)\to H^p(M_{\tau}),\ p\in \mathbb N,$ has dense image. 
	\end{lemma}
	We refer the readers to \cite[Appendix B]{SX3} for a proof of Lemma \ref{LmHV} for $p=1$, and the $p>1$ case is similar.

	Now we construct the small initial perturbation. Choose
	\[
	q_\ell(y)=c_{\mathrm H}\sum_{i=1}^{\ell}h_2(y_i),
	\]
	where $c_{\mathrm H}>0$ is chosen so that $\mathsf M(q_\ell)=\operatorname{diag}(I_\ell,0)$. Apply Lemma \ref{LmHV} to find initial data $\zeta_0$ for the variational equation such that, for a fixed small $\xi>0$, the pullback $\bar\zeta(T)$ satisfies
	\[
	\left\|\PiZ(\chi\bar\zeta(T))-q_\ell\right\|_{H^1}
	\leq \frac{\xi}{3}\|q_\ell\|_{H^1},
	\qquad
	\|\PiNZ(\chi\bar\zeta(T))\|_{H^1}
	\leq \frac{\xi}{3}\|q_\ell\|_{H^1}.
	\]
	After rescaling $\zeta_0$, this normalization entails no loss of generality.
	
	Next, we use the variational equation to approximate the RMCF. Suppose $\widetilde{M}^\eta_{\tau}$ is the RMCF starting from $\{x+\eta \zeta_0\bn(x)\ |\ x\in M_0\}$. Then $\widetilde{M}^\eta_T$ can be written as the graph of a function $\eta\tilde{\zeta}(T)$ over $M_T$, and $\|\eta\tilde{\zeta}(T)-\eta \bar\zeta(T)\|_{C^2}\leq C\eta^{1+\sigma_0}$ for some $\sigma_0>0$ (Proposition 3.3 of \cite{SX1}). We study the behavior of the solutions to the equation \eqref{EqODEDense} with this perturbed setting. Note that with the choice of the initial perturbation, the difference function $w$ satisfies 
	\[
	\|w(\cdot,T)\|_{H^1}=C\eta+O(\eta^{1+\sigma_0}).
	\]
	This also implies that there exists $C>0$ such that
	\[
	(1-C\xi-C\eta^{\sigma_0})\eta I_\ell
	\leq A_1(T)\leq
	(1+C\xi+C\eta^{\sigma_0})\eta I_\ell,
	\qquad
	|A_2(T)|+|A_{12}(T)|\leq C(\xi+\eta^{\sigma_0})\eta.
	\]

	\subsection{Definite time behavior of the perturbed flow}\label{SS_definite time perturbed}
	
	In this section, we suppose $K>0$ is a sufficiently large, but fixed constant; $T$ is a sufficiently large, fixed time. We choose the initial perturbation $\{x+\eta w_0\bn(x)\ |\ x\in M_0\}$ given by the previous section.
	
	\begin{lemma}\label{LmE}
		For every fixed $c_*>0$, if $K$ is sufficiently large, the projected cutoff error in \eqref{EqODEDense} satisfies
		\begin{equation}\label{EqIntegratedCutoff}
			\int_T^{T+c_*\eta^{-1}}|\Err(\tau)|\,d\tau
			\leq C e^{-K^2T/8}\eta+C\eta^{10}
		\end{equation}
		on the interval on which the graphical bootstrap holds.
	\end{lemma}
	
	\begin{proof}[Proof of Lemma \ref{LmE}]
		Proposition \ref{PropNeutralProjection} gives, after absorbing the polynomial factor into the Gaussian,
		\[
		|\Err(\tau)|\leq
		C e^{-K^2\tau/8}\|w(\tau)\|_{C^2(B_{K\sqrt\tau})}.
		\]
		Let $s_0=C_1^{-1}\log(\dt_1/(C\eta))$. Lemmas \ref{LmCM313} and \ref{lem_C2 of diff bound by C0} give $\|w(T+s)\|_{C^2}\leq C\eta e^{C_1s}$ for $0\leq s\leq s_0$. If $K^2/8>C_1$, integration on this interval is bounded by $Ce^{-K^2T/8}\eta$. On the remaining interval, the crude pseudolocality estimate $\|w(T+s)\|_{C^2}\leq Ce^{(T+s)/2}$ suffices. Since $e^{-(K^2/8-1/2)s_0}$ is a power of $\eta$, choosing $K$ so that $(K^2/8-1/2)/C_1>10$ bounds the remaining integral by $C\eta^{10}$. This proves \eqref{EqIntegratedCutoff}.
	\end{proof}
	
	Later, we will show that on the time interval $\tau\in [T,T+O(\eta^{-1})],$ $a_1$ is always of order $\eta$, thus the contribution of the error $\mathcal E$ is negligible for all time under consideration. With this estimate, we can estimate $a_1$ up to time of order $\eta^{-1}$. At that final time, the difference function $w$ is of order $1/\tau$, comparable to the graphical function $u$ for the unperturbed flow. 
	
	\begin{lemma}\label{LmSolODE}
		Suppose the coefficient system \eqref{EqODEDense}, with the error bounds
		\eqref{eq:Qbound} and \eqref{EqIntegratedCutoff}, is valid on $[T,T+c_*\eta^{-1}]$. Suppose
		\[
		c_-\eta I\leq A_1(T)\leq c_+\eta I,
		\qquad
		|A_2(T)|+|A_{12}(T)|\leq\varepsilon_0\eta,
		\]
		where $0<c_-<c_+<\infty$ and $0<\varepsilon_0\ll1$. Assume also that the neutral cone condition
		\begin{equation}\label{EqNeutralCone}
			\|\Pi_{\ne0}(\chi w)(\tau)\|_{H^1}
			\leq \tau^{-\vartheta}\|\Pi_0(\chi w)(\tau)\|_{H^1},
			\qquad \vartheta>\frac45,    
		\end{equation}
		holds throughout this interval. Let $\gamma>0$ as in Lemma \ref{lem:neutral-block-ode}, and let $\sigma>0$ be as in \eqref{eq:Qbound}. After decreasing $c_*$ and $\varepsilon_0$, every eigenvalue $a_j$ of $A_1$ stays positive and, for $0\leq s\leq c_*\eta^{-1}$,
		\begin{equation}\label{EqSolaj}
			a_j(T+s)=
			\frac{a_j(T)}{1+\gamma a_j(T)s}
			\bigl(1+O(\varepsilon_0+T^{-\sigma})+o_\eta(1)\bigr),
		\end{equation}
		and
		\begin{equation}\label{EqA2<A1}
			|A_2(T+s)|+|A_{12}(T+s)|
			\leq C\varepsilon_0\,|A_1(T+s)|.
		\end{equation}
	\end{lemma}
	
	\begin{proof}
		
		Let $s_*$ be the first elapsed time at which either an eigenvalue of $A_1$ leaves the interval
		\[
		\left[\frac{c_-\eta}{2(1+\gamma c_+\eta s)},
		\frac{2c_+\eta}{1+\gamma c_-\eta s}\right]
		\]
		or the right-hand side of \eqref{EqA2<A1}, with $C$ replaced by $2C$, is attained. At almost every differentiability time in $[0,s_*]$, the ordered eigenvalues satisfy
		\begin{equation}\label{EqEigenDense}
			\dot a_j=-\gamma a_j^2
			+O\bigl((\varepsilon_0+T^{-\sigma})a_j^2\bigr)
			+O\bigl((T+s)^{-1-\sigma}a_j\bigr)+e_j(T+s).
		\end{equation}
		The same inequalities hold in the one-sided sense at multiple eigenvalues, and \eqref{EqIntegratedCutoff} bounds the integral of $|e_j|$. Integrating the scalar Riccati inequalities, using $\int_T^\infty\tau^{-1-\sigma}d\tau=O(T^{-\sigma})$, and first choosing $K,T$ so that $e^{-K^2T/8}\ll\varepsilon_0$ gives \eqref{EqSolaj} uniformly on $[0,s_*]$, with strict improvement of the eigenvalue barriers.
		
		For $A_2$ and $A_{12}$, variation of constants in \eqref{EqODEDense}, together with \eqref{eq:Qbound}, \eqref{EqNeutralCone}, and \eqref{EqSolaj}, yields
		\[
		|A_2(T+s)|+|A_{12}(T+s)|
		\leq \frac32 C\varepsilon_0\|A_1(T+s)\|.
		\]
		The integrated cutoff contribution is absorbed by the same choice of $K,T$, while $\eta^{10}=o_\eta(1)\eta$. Thus neither exit condition can occur at $s_*<c_*\eta^{-1}$. Hence $s_*=c_*\eta^{-1}$, proving \eqref{EqSolaj} and \eqref{EqA2<A1}.
	\end{proof}
	
	\subsection{The further perturbation argument}\label{SS:PerturbCondormal}
	
	We perform a family of spacetime translations and dilations of the perturbed initial manifold $\widetilde M_0$. After a fixed nonsingular linear reparametrization, the positive spectral coefficients of the difference graphs at time $T$ form the family
	\[
	\mathcal W=
	\left\{\eta\left(q_\ell+\sum_{j=1}^ks_jy_j+s_0\right):
	(s_0,s_1,\ldots,s_k)\in[-\xi,\xi]^{k+1}\right\},
	\]
	up to the errors estimated in the preceding subsection. For a parameter $\mathbf s$, let $w^{\mathbf s}(\tau)$ be the corresponding difference graph and use the cutoff $\chi_{K\sqrt\tau}$. Smooth dependence of the mean curvature flow gives continuous dependence of $\chi_{K\sqrt\tau}w^{\mathbf s}(\tau)$ on $\mathbf s$ on the common graphical bootstrap interval.
	
	We now make the selection step precise. Let $D=[-\xi,\xi]^{k+1}$ and $S=T+c_*\eta^{-1}$. Consider the first time at which a member of the family reaches the boundary of the neutral cone \eqref{EqNeutralCone}. The one-step cone estimates in Theorem \ref{ThmCone}, together with \eqref{EqIntegratedCutoff}, make such a crossing strict and outward when $T,K$ are large and $\eps_0\ll\xi$. On $\partial D$, the normalized coefficient vector of $1,y_1,\ldots,y_k$ is a small perturbation of $(s_0,s_1,\ldots,s_k)$; these are precisely the coefficients generated by the chosen dilation and translations. If every parameter exited before $S$, the first-exit coefficient vector would consequently define a continuous retraction of $D$ onto $\partial D$ whose boundary restriction has degree one. This is impossible. Hence some $\mathbf s_*\in D$ stays in the neutral cone up to time $S$. The contracting alternative in Theorem \ref{ThmCone} controls its negative modes, so \eqref{EqNeutralCone} holds throughout $[T,S]$.
	
	Lemma \ref{LmSolODE} applies to $w^{\mathbf s_*}$. At time $S$,
	\[
	\frac{c}{S}I_\ell\leq A_1(S)\leq\frac{C}{S}I_\ell,
	\qquad |A_2(S)|+|A_{12}(S)|\leq\frac{C\varepsilon_0}{S}.
	\]
	The unperturbed matrix on the remaining block is $(\gamma S)^{-1}I_{k-\ell}+O(S^{-1-\vartheta})$. Thus, after decreasing $\varepsilon_0$, the full neutral matrix of the selected flow is positive definite and all its eigenvalues lie between $c/S$ and $C/S$.
	
	Starting at $S$, the centering argument of Section \ref{S_stability} applies with this uniformly positive definite coefficient matrix (the degree proof of Proposition \ref{PropCentering} is unchanged, since the
	axial translation block is then invertible). Pseudolocality and the Ornstein--Uhlenbeck estimate continue the graphical bounds on $B_{\tau^\kappa}$. The matrix estimate \eqref{EqM} and the spectral bounds at time $S$ give, by the standard eigenvalue comparison, $\lambda_{\min}(\tau)\geq c/\tau$ for $\tau\geq S$. Lemma \ref{LmNeutralPersistence} therefore shows that the limiting projection is $I_k$. Hence the selected nearby flow has a nondegenerate cylindrical singularity. This completes the proof of the denseness theorem.
	
	\begin{proof}[Proof of Theorem \ref{ThmGenericity}]
		Combining all the ingredients of this section proves Theorem \ref{ThmGenericity}.
	\end{proof}

	\section{Applications}\label{SApplication}
	In this section, we give some applications of our main result to level set flows, rotational graphs, etc. 
	\subsection{Regularity of level set flow}
	
	In this section, we prove Theorem \ref{thm:main thm LSF}. We first prove the following lemmas:
	
	\begin{lemma}\label{lem:2times}
		If $\{\bM_t\}$ is an MCF with a nondegenerate singularity that is modeled by $\cC_{n,k}$ with $k\geq 1$, then $\{\bM_{t}\}$ has at least two distinct singular times.
	\end{lemma}
	
	\begin{proof}
		Without loss of generality, assume $(0,0)$ is the nondegenerate singularity of $\{\bM_{t}\}$. We show that $\bM_{t}$ is nonvanishing after time $0$. Recall that in \cite[Theorem 1.6]{SunXue25_GenericI}, we prove that there exists $\delta>0$ such that $\bM_{t}\cap (B_{\delta}(0)\backslash\{0\})$ is smooth for $t\in[-1,0]$ for some $\delta>0$. This shows that $\bM_{t}$ is nonvanishing after time $0$, and hence $\bM_{t}$ has at least two distinct singular times.
	\end{proof}
	
	In order to apply our perturbation scheme, we need to require the flow to be type-I. This is true when the flow has the singular set to be a space $C^1$-submanifold.
	
	\begin{lemma}\label{lem_C2_type-I}
		If $\{\bM_t\}$ is a MCF with the following significance:
		\begin{enumerate}
			\item There is exactly one singular time $T$, where the flow becomes extinct;
			\item The singular set $\cS$ is a $k$-dimensional closed, connected, embedded $C^1$ submanifold of singularities where the tangent flow is $\cC_{n,k}$ at each point.
		\end{enumerate}
		Then $\{\bM_t\}$ is type-I.
	\end{lemma}
	
	\begin{remark}
		Colding-Minicozzi \cite{CM5} proved that the two assumptions in Lemma \ref{lem_C2_type-I} are equivalent to the arrival time function being $C^2$ if the flow has cylindrical singularities.
	\end{remark}
	
	\begin{proof}
		The proof is essentially contained in \cite{CM3}, and we only sketch the proof here. Let us use $PB_r(y)$ to denote the parabolic ball centered at $y$ with radius $r$. A flow $\bM_t$ is $(j, \eta)$-cylindrical at $y$ on the time-scale $s_0$ if for every positive $s<s_0$, $\frac{1}{\sqrt{s}}([B_{\eta^{-1}\sqrt{s}} (y) \cap \bM_{t-s}]-y)$ is a graph over a $\cC_{n,k}$, after possible rotations, of a function with $C^1$ norm at most $\eta$. Note that by the parabolic regularity, this definition also implies a $C^2$-graphical bound at a slightly later time.

		By Corollary 3.2 in \cite{CM3}, for any $\eta > 0$, $j>0$ and $y$ in the singular set, there exist $r_y>0$ and $t_y > 0$ so that $M_{\tau}$ is uniformly $(j,\eta)$-cylindrical on the parabolic region $\cS\cap PB_{r_y}(y)$ on the time-scale $t_y$. 
		
		Without loss of generality, assume the singular time is $0$. For each $y\in\cS$, we can find an open neighborhood $U_y$ of it so that after time $-t_y$, $M_{\tau}$ is uniformly $(j,\eta)$-cylindrical on $\cS\cap U_y$, and hence $\sqrt{-t}|A|<C$ for $-t_y<t<0$ inside $U_y$. Then the type-I condition follows from a compactness argument. 
	\end{proof}
	
	\begin{proof}[Proof of Theorem \ref{thm:main thm LSF}]
		Let us define $\cU_1$ to be the set of mean convex hypersurfaces that have a single spherical singularity, and the spherical singularity is the most generic one in the sense of \cite{SX3}. We define $\cU_2$ to be the set consisting of mean convex closed hypersurfaces such that the level set flows starting from these hypersurfaces have a nondegenerate cylindrical singularity. 
		
		The openness of $\cU_1$ is proved in \cite{SX3}, and the openness of $\cU_2$ is proved by the stability of the nondegenerate cylindrical singularity Theorem \ref{ThmStability}. Hence $\cU=\cU_1\cup\cU_2$ is open. Next, we prove denseness. Given a mean convex hypersurface $\Sigma$, the first-time singularity is either spherical or cylindrical. If the first-time singularity is spherical, from \cite{SX3}, $\Sigma\in\overline{\cU_1}$. If the first-time singularity is cylindrical, then Lemma \ref{lem_C2_type-I} shows that the flow is type-I, and we can apply the genericity Theorem \ref{ThmGenericity} to show that $\Sigma\in\overline{\cU_2}$. Then we proved the denseness of $\cU$.
		
		Finally, for hypersurfaces in $\cU_1$, the level set flow starting from them has $C^2$ but not $C^3$ regularity argued as \cite{SX3}; for hypersurfaces in $\cU_2$, the level set flow starting from them has at least two distinct singular time by Lemma \ref{lem:2times}, hence by \cite{CM5} the level set flow has $C^{1,1}$ but not $C^2$ regularity. Then the proof is concluded.
	\end{proof}

	\subsection{Pinch a thin ring}\label{SMarriageRing}
	
	In this section, we prove the global Theorem \ref{ThmMarriageRing} on pinching a thin ring.

	Without loss of generality, we assume $\gamma$ is a curve passing through the origin, and we will pinch the MCF at the origin. Let us summarize some known facts about MCF whose singular set is a smooth curve. By the stratification of the singular set (\cite{CheegerHaslhoferNaber13}), the tangent flow at any singular point splits along a line, hence must be a cylinder in $\R^3$. By Colding-Minicozzi's analysis of cylinder tangent flow, \cite{CM3}, as the time approaches the singular time, the flow becomes a ``tube'' around the singular set. Also, by Lemma \ref{lem_C2_type-I}, the flow has the type-I curvature condition.
	
	\begin{proof}
		Let us construct the perturbation at a large time $T$. We consider the RMCF $M_{p,\tau}=e^{\tau/2}(\mathbf M_{-e^{-\tau}}-p)$ blow up at each point $p\in \gamma$. 
		
		We first find a large time $T$ that is uniform for all $p$, such that the at time $T$, $M_{p,T}$ can be written as the graph over $\Sigma^1$ of a function $u_p:\ \Sigma^1\cap B_{\mathbf r(T)}\to \R$ with $\|u_p\|_{C^2}\leq \eta. $ The uniformity of $T$ is given by \cite{CM2} and \cite[Corollary 3.2]{CM3}. 
		
		We next construct a function $f:\bM_{-e^{-T}}\to\R$ defined as follows: 
		\begin{itemize}
			\item Inner region: Inside $B_{e^{-T/2}(\mathbf{r}(T))}$, let $f$ be the transplantation of $\eps_0\sqrt{1+\frac{y^2-2}{T}}$ over $\bM_{-1}\cap B_{\mathbf{r}(T)}$. 
			\item Outer region: Outside $B_{e^{-T/2}(\mathbf{r}(T)+1)}$, let $f$ be the constant $u_O=e^{-T/2}(\mathbf{r}(T)^2-2)$.
			\item Transition region: In the annulus region $B_{e^{-T/2}(\mathbf{r}(T)+1)}\backslash B_{e^{-T/2}(\mathbf{r}(T))}$, let $f$ be the smooth interpolation.
		\end{itemize}

		Let $\eps_1>0$ be a sufficiently small number much smaller than $\eps_0$. We next apply Lemma \ref{LmHV} of Herrero-Vel\'azquez to the linearized MCF equation: $\partial_t u=\Delta_{\bM_{t}}u+|A|^2u$ on the time interval $t\in [-1,-e^{-T}]$ to find a function $u(-1)$ such that $\|u(-e^{-T})-f\|_{C^2}\leq \eta$. 
		
		With $u(-1)$, for any small $\eta>0$ we introduce the perturbed MCF $(\widetilde \bM^\eta_t)$ with initial condition $\widetilde \bM^\eta_{-1}:=\mathrm{Graph}\{x+\eta u(-1,x)\mathbf n(x)\ | \ x\in \bM_{-1}\}. $
		
		We next prove that for sufficiently small $\eta>0$, the perturbed MCF $\widetilde\bM^\eta_{-1}$ verifies the statement of the theorem. First, we cite the following Pseudolocality result from \cite[Lemma 7.1]{SunXue25_GenericI}. 
		\begin{lemma}[Pseudolocality of cylindrical MCF over $\cC_{n,k}$ with a slightly larger radius]\label{lem:plcyl}
			For any $\eps_0>0$, and a vector $V\in \R^k$ with $0<|V|<\eps_0$, there exist $T_0>0$, $R_1>0$ and $\eps_1>0$ with the following significance. Suppose $T>T_0$, $\bM_t$ is a mean curvature flow such that, $\bM_0$ is the graph of a function $u$ over the cylinder $\cC_{n,k}$ inside a ball of radius $\eps_0\sqrt{T}$, with $\|u(\theta,y)-C|T^{-1/2}y-V|^2\|_{C^1}<\eps_1$, where $C>0$ is some constant. Then for $t\in[0,1]$, $\bM_t\cap B_{R_1}$ is a smooth piece of the mean curvature flow.
		\end{lemma}
		Theorem 1.7 (nondegenerate singularity has an isolated neighborhood) of \cite{SunXue25_GenericI} proves that the nondegenerate singularity has an isolated neighborhood $U_p$ free of singularities. Lemma \ref{lem:plcyl} applied in the same way as Theorem 1.7  of \cite{SunXue25_GenericI} shows that the transition region has no singularities. Theorem \ref{thm_existence_nondeg_with_ansatz} implies that there is a nondegenerate singularity around $p$. 
		
		We next show that there is no singularity outside $U_p$. For the perturbation at time $T$ of the RMCF, in this case, instead of writing the graph of the function $f=u_O$ over the standard cylinder $\mathcal C_{n,k}$, we write it as the graph of $f=0$ over $(1+\frac{u_O}{\sqrt{2(n-k)}})\mathcal C_{n,k}.$ Thus, we apply the above pseudolocality Lemma \ref{lem:plcyl} to show that the flow outside $U_p$ does not develop a singularity within time $1$. 
		
	\end{proof}

	\appendix
	\section{Transformation of graphs}
	
	We need the following properties of the translations and dilations of graphs in this paper. Most of the consequences are routine calculations. In the following, we fix sufficiently small $\eps_0>0$ and $\delta_0>0$. 
	
	\begin{proposition}\label{prop:AppGraph_TransDila}
		There exists $\kappa_n''\in (0, 1/2)$ such that if $\delta\in (0, \kappa_n'']$, $(\bx, \by)\in \RR^{n-k+1}\times \RR^k$, $\mbfA\in \mfk g_{n,k}^\perp$, $\lambda>0$, $R>2n$ satisfy, \[
		\|u\|_{C^2(Q_R(0, \lambda^{-1}\by))} + |\bx| + |\by| + R(|\mbfA| + |\lambda - 1|) \leq \delta \,,
		\]
		then in $Q_{R - C_n\delta}$, $e^\mbfA(\lambda\Sigma - (\bx, \by))$ is also a graph over $\cC_{n,k}$, and the graphical function $\bar u$ satisfies, 
		\begin{equation}\label{eq:AppE(v)}
			\begin{split}
				\sup_{(\theta', y')\in \cC_{n,k}\cap Q_{R-C_n\delta}} &\, \left|\bar u(\theta', y') - u(\theta', \lambda^{-1}\by +y') - \sqrt{2(n-k)}(\lambda-1) + \psi_\bx(\theta') - \psi_\mbfA(\theta', y')\right| \\
				& \leq C_n \left(\|u\|_{C^2(Q_R(0, \by))} + R|\lambda-1| + |\bx|+ R|\mbfA| \right)\cdot (R|\lambda-1| + |\bx| + R|\mbfA|)\,.  
			\end{split}
		\end{equation}
		
	\end{proposition}
	
	\begin{proof}
		The proof is based on the routine calculation and the mean value theorem.In fact, let $\hat\theta$ be the normalized vector of $\theta$. The coordinate identity corresponding to the transformation in the statement is
		\[
		e^{-\mbfA}\bigl(\theta'+\bar u(\theta',y')\hat\theta',y'\bigr)
		+(\bx,\by)
		=\lambda\bigl(\theta+u(\theta,y)\hat\theta,y\bigr).
		\]
		Comparing its spherical and axial components gives the desired estimates. We omit the remaining calculation details here.
	\end{proof}
	
	We also need the following lemma proved in \cite{SunWangXue1_Passing}
	
	\begin{lemma} 
		Let $\zeta\in L^\infty(\cC_{n,k})$ be such that $0\leq \zeta\leq 1$ and $\zeta|_{Q_R} = 1$; $\Pi: L^2(\cC_{n,k})\to L^2(\cC_{n,k})$ be the orthogonal projection map onto a closed linear subspace. Then for every $v\in L^2(\cC_{n,k})$, we have \[
		\big|\|\Pi(v\zeta)\|_{L^2} - \|\Pi(v)\|_{L^2}\big| \leq 2\|v\|_{L^2(\cC_{n,k}\setminus Q_R)} 
		\]
		In particular, if $\Lambda>0$, $v\in \oplus_{\gamma\leq \Lambda} \rmW_\gamma$, then \[
		\big|\|\Pi(v\zeta)\|_{L^2} - \|\Pi(v)\|_{L^2}\big| \lesssim_{n, \Lambda} e^{-R^2/10}\|v\|_{L^2}\,.
		\]
	\end{lemma}
	
	Now let us consider a compactly supported function  $u :\cC_{n,k}\to\R$, that is supported in $B_R$ with $\|u\|\leq \delta$. Let the graph of $u$ be $\Sigma$. Let $\xi:=|\bx| +|\by| + R(|\mbfA| + |\lambda - 1|)< \delta$ and let $\bar u$ be the graph of $e^{\mbfA}(\lambda\Sigma-(\bx,\by))$. Then for any unit eigenfunction $\psi$ we have
	\begin{align*}
		\int \chi_{R-C_n\delta}\psi \bar u d\mu
		=&
		\int \chi_{R-C_n\delta}  \psi \left(
		u(\theta',\lambda^{-1}\by+y') 
		+\varrho(\lambda-1)
		-\psi_\bx(\theta')
		+\psi_\mbfA(\theta',y')
		+\text{error}
		\right)d\mu
		\\
		=&
		\langle \psi,u\rangle+\varrho(\lambda-1)\langle \psi,1\rangle 
		-
		\langle \psi,\psi_\bx\rangle
		+
		\langle \psi,\psi_{\mbfA}\rangle
		\\
		&+O(e^{-(R-C_n\delta)^2/5}(\|\bar u\|_{L^2}+\xi))
		+O\left( \left(\|u\|_{C^2(Q_R(0, \by))} + \xi \right)\xi\right).
	\end{align*}
	
	Here we use the following fact: for any two functions $f,g$, using the mean value theorem,
	\begin{align*}
		\int f(\theta',y')g(\theta',\lambda^{-1}\by+y')d\mu(\theta',y')
		=&
		\int f(\theta',y')g(\theta',y') d\mu
		+
		O(\|f\|_{L^2}|\nabla g|_{C^0}|\lambda^{-1}\by|)
	\end{align*}
	
	From the above analysis, we see that $\psi_\by$ does not show up when we apply the translation. A key observation in the stability analysis of the nondegenerate singularities is that if the term $\sum_{i=1}^n(y_i^2-2)$ shows up, then $\psi_\by$ will show up when we apply the translation. In fact, we have the following corollary of Proposition \ref{prop:AppGraph_TransDila}.
	
	\begin{corollary}\label{cor_change_TDR}
		Suppose all the assumptions in Proposition \ref{prop:AppGraph_TransDila}. Further suppose $u=a \sum_{i=1}^k(y_i^2-2)+v$ and $\bar u=a \sum_{i=1}^k(y_i^2-2)+\bar v$. Then 
		\begin{equation}\label{eq:AppE(v)_nondeg}
			\begin{split}
				\sup_{(\theta', y')\in \cC_{n,k}\cap Q_{R-C_n\delta}} &\, \big|\bar v(\theta', y') - v(\theta', \lambda^{-1}\by +y') - [\sqrt{2(n-k)}(\lambda-1)+a\lambda^{-2}|\by|^2] 
				\\
				&\quad +
				\psi_\bx(\theta') - \psi_\mbfA(\theta', y')
				-2a\lambda^{-1}\psi_\by(y')
				\big| \\
				& \leq C_n \left(\|u\|_{C^2(Q_R(0, \by))} + R|\lambda-1| + |\bx|+ R|\mbfA| \right)\cdot (R|\lambda-1| + |\bx| + R|\mbfA|)\,.  
			\end{split}
		\end{equation}
		As a consequence,
		Let $\xi:=|\bx|+|\by| + R(|\mbfA| + |\lambda - 1|)< \delta$ and let $\bar u$ be the graph of $e^{\mbfA}(\lambda\Sigma-(\bx,\by))$. Then for any unit eigenfunction $\psi$ we have
		\begin{equation}
			\begin{split}
				\int \chi_{R-C_n\delta}\psi \bar u d\mu
				=&
				\langle \psi,u\rangle+[\varrho(\lambda-1)+a\lambda^{-2}|\by|^2]\langle \psi,1\rangle 
				-
				\langle \psi,\psi_\bx\rangle
				+
				\langle \psi,\psi_{\mbfA}\rangle
				+
				2a\lambda^{-1}\langle \psi,\psi_\by\rangle
				\\
				&+O(e^{-(R-C_n\delta)^2/5}(\|\bar u\|_{L^2}+\xi))
				+O\left( \left(\|u\|_{C^2(Q_R(0, \by))} + \xi \right)\xi\right).
			\end{split}
		\end{equation}
	\end{corollary}

	\bibliographystyle{alpha}
	\bibliography{GMT}
\end{document}